\documentclass[10pt,reqno]{amsart}

\usepackage{amsmath,amscd,amssymb,amsthm}
\usepackage{tikz-cd}
\usepackage{comment}

\makeatletter
\def\th@plain{%
  \thm@notefont{}
  \itshape 
}
\def\th@definition{%
  \thm@notefont{}
  \normalfont 
}
\makeatother

\usepackage{enumitem}
\setlist[enumerate]{label=(\roman*),leftmargin=0.8cm}
\usepackage{todonotes}
\usepackage{color}

\newtheorem{proposition}{Proposition}[section]
\newtheorem{lemma}[proposition]{Lemma}
\newtheorem{theorem}[proposition]{Theorem}
\newtheorem{corollary}[proposition]{Corollary}

\theoremstyle{definition}
\newtheorem{remark}[proposition]{Remark}
\newtheorem{definition}[proposition]{Definition}

\newtheorem{example}[proposition]{Example}

\numberwithin{equation}{section} 

\DeclareMathOperator{\Bl}{Bl}

\DeclareMathOperator{\Aut}{Aut}
\DeclareMathOperator{\Ric}{Ric}

\DeclareMathOperator{\V}{V}

\DeclareMathOperator{\Lie}{Lie}

\DeclareMathOperator{\BS}{BS}

\renewcommand{\L}{\mathcal{L}}
\renewcommand{\phi}{\varphi}

\DeclareMathOperator{\DF}{DF}

\DeclareMathOperator{\Hom}{Hom}

\newcommand{\R}{\mathbb{R}}
\newcommand{\C}{\mathbb{C}}

\newcommand{\pr}{\mathbb{P}}
\renewcommand{\epsilon}{\varepsilon}
\renewcommand{\H}{\mathcal{H}}
\newcommand{\D}{\mathcal{D}}

\newcommand{\scO}{\mathcal{O}}

\newcommand{\X}{\mathcal{X}}

\newcommand{\ddb}{i\partial \bar\partial}

\newcommand{\scY}{\mathcal{Y}}
\newcommand{\mfk}{\mathfrak{k}}
\newcommand{\mft}{\mathfrak{t}}

\newcommand{\scB}{\mathcal{B}}

\renewcommand{\V}{\mathcal{V}}

\newcommand{\Y}{\mathcal{Y}}
\newcommand{\scA}{\mathcal{A}}

\newcommand{\I}{\mathcal{I}}

\DeclareMathOperator{\Vol}{Vol}

\DeclareMathOperator{\wt}{wt}
\newcommand{\E}{\mathcal{E}}

\newcommand{\ddbar}{\partial\bar{\partial}}
\DeclareMathOperator{\Fut}{Fut}
\DeclareMathOperator{\sq}{wt^2}
\DeclareMathOperator{\Hess}{Hess}

\title[Extremal K\"ahler metrics on blowups]{Extremal K\"ahler metrics on blowups}

\author[Ruadha\'i Dervan and Lars Martin Sektnan]{Ruadha\'i Dervan and Lars Martin Sektnan}
\address{Ruadha\'i Dervan, School of Mathematics and Statistics, University of Glasgow, University Place, Glasgow G12 8QQ, United Kingdom}\email{ruadhai.dervan@glasgow.ac.uk}

\address{Lars Martin Sektnan, Department of Mathematical Sciences, University of Gothenburg, 412 96 Gothenburg, Sweden and Institut for Matematik, Aarhus University, 8000, Aarhus C, Denmark}
\email{sektnan@chalmers.se}

\begin{document}

\begin{abstract} Consider a compact K\"ahler manifold which either admits an extremal K\"ahler metric, or is a small deformation of such a manifold. We show that the blowup of the manifold at a point admits an extremal K\"ahler metric in K\"ahler classes making the exceptional divisor sufficiently small if and only if it is relatively K-stable, as predicted by the Yau--Tian--Donaldson conjecture. We also give a geometric interpretation of what relative K-stability means in this case in terms of finite dimensional geometric invariant theory.  This gives a complete solution to a problem introduced and solved by Arezzo, Pacard, Singer and Sz\'ekelyhidi for constant scalar curvature K\"ahler metrics in dimension at least three. 
\end{abstract}

\maketitle

\section{Introduction}

A central goal of K\"ahler geometry is to understand the existence of canonical representatives of K\"ahler classes. The natural representatives are \emph{constant scalar curvature K\"ahler (cscK) metrics} and more generally \emph{extremal metrics}. Such metrics do not always exist, and the Yau-Tian-Donaldson conjecture states that the existence of cscK metrics should be equivalent to the algebro-geometric notion of \emph{K-stability} \cite{yau,tian,donaldson}. Similarly the existence of extremal metrics should be equivalent to \emph{relative K-stability} \cite{gabor-extremal}. Despite significant progress, this conjecture is open in general. Furthermore, even when the conjecture is known to hold, the geometric meaning of K-stability is typically unclear.

One of the first constructions of cscK metrics is due to Arezzo-Pacard \cite{arezzo-pacard1,arezzo-pacard2}, who proved results relating to the existence of cscK metrics on blowups of manifolds known to admit cscK metrics. In the absence of automorphisms of the starting manifold, they construct cscK metrics on the blowup using a gluing method. Perhaps the most interesting aspect of their work is that when the starting manifold admits automorphisms, there are algebro-geometric obstructions in their gluing argument to obtaining cscK metrics on the blowup; the obstructions are related to stability of the blown-up point in the sense of geometric invariant theory. Significantly, this was the first general construction of cscK metrics in which algebro-geometric stability enters into the analysis. The problem and analogous results were generalised to the extremal setting by Arezzo--Pacard--Singer \cite{APS}. 

An important problem in the field has since been to characterise the existence of extremal metrics on the blowup through relative K-stability, in line with the Yau-Tian-Donaldson conjecture, and we refer to Pacard \cite{pacard} and Sz\'ekelyhidi  \cite{gabor-icm} for surveys on this problem and for further context. Sz\'ekelyhidi has made substantial progress on this problem, including a complete solution in the cscK case provided the complex dimension is at least three \cite{gabor-blowups1,gabor-blowups2}. Sz\'ekelyhidi's strategy is to produce very strong approximate solutions to the cscK equation on the blowup, and to show that the higher order terms in these approximate solutions can be matched to zeroes of certain moment maps on $X$ itself, which is how geometric invariant theory on $X$ enters. The analysis involved has resisted progress beyond this case, meaning in dimension two (see Datar for work in this direction \cite{datar}), and in the extremal case, which are both open.

We take a new approach to the problem, with which we provide a complete solution in general. Essentially, in the prior approaches one must obtain better and better approximate solutions to the equation of interest on the blowup. Our new strategy avoids this by makes the geometry itself more involved, with the benefit of significantly simplifying the analysis.

To state the main results precisely we require some further notation. Consider a compact K\"ahler manifold $X$, together with a K\"ahler class $\alpha$ which admits an extremal metric $\omega\in\alpha$. For a point $p \in X$, consider the blowup $\sigma: \Bl_p X \to X$ endowed with the K\"ahler class $\alpha_{\epsilon} = \sigma^*\alpha - \epsilon^2 [E]$ with $E$ the exceptional divisor. We fix a maximal compact subgroup $K \subset \Aut_0(X,\alpha)$ and a maximal torus $T\subset K_p$, where $K_p$ is the stabiliser of $p$ in $K$.  Denoting by $K^T$ the centraliser of $T$ in $K$, we then consider a family of moment maps (with $\Delta$ the Laplacian) $$A_{\epsilon}\mu + B_{\epsilon} \Delta\mu: X \to (\mfk^T)^*$$ for the $K^T$-action on $X^T$ (the fixed locus of $T$) with respect to $A_{\epsilon} \omega + B_{\epsilon} \Ric \omega$, where $A_{\epsilon}, B_{\epsilon}$ are functions of $\epsilon$ defined explicitly in Corollary \ref{k-stabhence} with $A_{\epsilon}>0$ and with $B_{\epsilon}$ of strictly higher order in $\epsilon$. We in addition define inner products $\langle \cdot, \cdot \rangle_{\epsilon,q}$ on $\mfk^T = \Lie K^T$ that depend on both $\epsilon$ and $q\in X^T$, and which have an explicit algebro-geometric interpretation.  Our main results can be summarised as follows:

\begin{theorem}\label{intromainthm-stable}
There is an $\epsilon_0>0$ such that for all $\epsilon \in (0, \epsilon_0)$ the following are equivalent:
\begin{enumerate}
\item $(\Bl_p X, \alpha_{\epsilon})$ admits an extremal metric;
\item $(\Bl_p X, \alpha_{\epsilon})$ is relatively K-stable;
\item for every element $u\in \mfk^T$ with $p$ specialising to $q$ and such that $u$ is orthogonal to $\mft=\Lie T$ under $\langle \cdot, \cdot \rangle_{\epsilon,q}$, we have $$A_{\epsilon}h(q) + B_{\epsilon}\Delta h(q)>0,$$ with $u$ having Hamiltonian $h$ with respect to $\omega$.
\end{enumerate}
\end{theorem}

The equivalence of $(i)$ and $(ii)$ proves the analogue of the Yau--Tian--Donaldson conjecture in this setting, while the equivalence with $(iii)$ further gives an explicit geometric interpretation of what relative K-stability means in this setting. When $X$ is projective with $\alpha = c_1(L)$ ample, $(iii)$ can further be understood in terms of completely classical geometric invariant theory. 

Theorem \ref{intromainthm-stable} is due to Sz\'ekelyhidi when both $n\geq 3$ and when one seeks to relate cscK metrics to K-stability; this sharp result is new in the remaining cases (so when either $\dim X=2$ or when one seeks extremal metrics in any dimension). Since the main novelty in our work is our new approach,  which shifts the difficulty from the analysis to the geometry, we outline the approach in detail in Section 2. Briefly, we begin by arguing in a universal manner: rather than a single point, we consider the blowup of all points of $X$ at once, by blowing up the diagonal in $X\times X$. This produces a holomorphic submersion over $X$, where the fibre over $p\in X$ is the blowup $\Bl_pX$; we further obtain a natural $\epsilon$-dependent family of relatively K\"ahler classes  on the total space of this family. 

Instead of involving moment maps on $X$ related to geometric invariant theory in the analysis, we use the moment map property of the scalar curvature directly. This makes the geometric setup more natural from the perspective of the Yau--Tian--Donaldson conjecture. More precisely, the holomorphic submersion structure produces an $\epsilon$-dependent sequence of moment maps on the base $X$ ---viewed as the base of this holomorphic submersion---where the moment map is (a finite-dimensional projection of) the fibrewise scalar curvature \cite{DH}, and there is a suitable variant of this statement in the extremal case due to Hallam. We then use the analysis involved in essentially the simplest case of the Arezzo--Pacard theorem (following the approach of Seyyedali--Sz\'ekelyhidi to argue in a more coordinate-free manner \cite{seyyedali-szekelyhidi}) to reduce the problem to finding a zero of these scalar-curvature moment maps on $X$, again also proving a  suitable variant of this in the extremal case. The point is then that the obstruction to solving the problem becomes (relative) K-stability, and so once the geometry is set up, we see that K-stability (respectively relative K-stability) implies the existence of a cscK (respectively extremal) metric directly and naturally.

This then leaves the task of geometrically interpreting K-stability in this specific situation, for which we follow the lines of Stoppa and Sz\'ekelyhidi \cite{stoppa, stoppa-szekelyhidi, gabor-blowups2}, extending their results to account also for the varying inner product. This produces a geometric interpretation of relative K-stability in terms of geometric invariant theory on $X$ itself, extending and recovering the prior results for cscK manifolds in dimension at least three. We emphasise that although we recover these prior results, in our approach the problem is solved in a different order. One main point of the Yau--Tian--Donaldson conjecture is that relative K-stability should be \emph{easier} to understand than the existence of extremal metrics, and our work is in line with this philosophy.

The techniques we develop are strong enough to also prove the ``semistable case'', which has not been considered before. Work of Stoppa and Stoppa-Sz\'ekelyhidi further implies that if $(X,\alpha)$ is relatively K-unstable, then its blowup is also relatively K-unstable in the classes we consider and hence cannot admit an extremal metric \cite{stoppa,stoppa-szekelyhidi}. Thus with the ``stable case'' settled in Theorem \ref{intromainthm-stable}, the only remaining case is that of a relatively K-semistable manifold.  We now consider a K\"ahler manifold $(X,\alpha)$ which is \emph{analytically relatively K-semistable}; this means that there is a suitably equivariant degeneration of $(X,\alpha)$ to an extremal K\"ahler manifold $(X_0,\alpha_0)$. As the name suggests, the condition implies relative K-semistability \cite{donaldson-lower,stoppa-szekelyhidi,relative}, and can also be seen as asking that $(X,\alpha)$ is a small equivariant deformation of an extremal manifold, so that it is an analytic version of relative K-semistability. We will require that $\alpha = c_1(L)$ for $L$ ample in order to appeal to a result requiring a form of compactness, so that $X$ is projective. 

The difference in the statement below in comparison with the stable case is the setup of the problem: instead of the action of the automorphism group  $(X,\alpha)$ itself, we consider the action of $\Aut_0(X_0,\alpha_0)^T$ on a space $\Y^T$ built from the Kuranishi space of $(X_0,\alpha_0)$ (blowing up the diagonal in a suitable fibre product), where $T$ is similarly a maximal torus lying in the stabiliser of the point $p \in X$ which fixes the extremal vector field on $X$ and $\X_0$ (in a sense which will be made precise after setting up the geometry explicitly). Thus the specialisation of $p$ under $u\in\mfk^T $ will no longer actually be a point on $(X,\alpha)$ itself in general. We prove the following:

\begin{theorem}\label{intromainthm-semistable}
There is an $\epsilon_0>0$ such that for all $\epsilon \in (0, \epsilon_0)$ the following are equivalent:
\begin{enumerate}
\item $(\Bl_p X, \alpha_{\epsilon})$ admits an extremal metric;
\item $(\Bl_p X, \alpha_{\epsilon})$ is relatively K-stable.\end{enumerate}\end{theorem}

We also obtain an analogue  of Theorem \ref{intromainthm-stable} $(iii)$ in the semistable case, involving stability of $p$ viewed as a point of $\Y$ in an explicit, geometric invariant theoretic sense. Although the resulting characterisation is quite technical, it can be understood in simple special cases. Supposing, for example, that the extremal degeneration of $(X,\alpha)$ has automorphism group isomorphic to $\C^*$, we show that Theorem \ref{intromainthm-semistable} implies that one can always find a point on $X$ such that its blowup admits an extremal metric. As we explain, this produces many new concrete examples of manifolds admitting a cscK or extremal metric.

The strategy in the proof of  Theorem \ref{intromainthm-semistable} is similar to that of Theorem \ref{intromainthm-stable}, once the geometry has been setup. We emphasise again that the advantage of our general approach is that the analysis is simplified, at the expense of making the geometry more involved; this is what makes Theorems \ref{intromainthm-stable} and \ref{intromainthm-semistable} tractable.

\subsection*{Acknowledgements} We thank Michael Hallam for explaining how to view the extremal equation as a moment map (described in Section \ref{moment-map-geometry}) to us, and G\'abor Sz\'ekelyhidi for a helpful discussion. RD was funded by a Royal Society University Research Fellowship. LMS was funded by a Marie Sk\l{}odowska-Curie Individual Fellowship, funded from the European Union's Horizon 2020 research and innovation programme under grant agreement No 101028041, and also by Villum Fonden Grant 0019098, while he was a member of Aarhus University. Part of work was completed while RD visited Newcastle University, and he thanks the department and Stuart Hall for their hospitality. 

\section{The main argument}

\subsection{Preliminaries} We recall the basic theory of  \emph{extremal K\"ahler metrics}, for which a reference is Sz\'ekelyhidi \cite{gabor-book}. We let $X$ be a compact K\"ahler manifold of dimension $n$, and let $\alpha$ be a K\"ahler class on $X$. For any K\"ahler metric $\omega \in \alpha$, its \emph{Ricci curvature} is denoted $\Ric \omega = -\frac{i}{2\pi} i\ddbar \log \omega^n,$ while its \emph{scalar curvature} is denoted $S(\omega) = \Lambda_{\omega}\Ric\omega.$

Defining the operator $\D = \bar\partial \nabla^{1,0}$ on functions, for $\omega$ to be extremal means that $\D S(\omega)=0.$ This means that the section $ \nabla^{1,0} S(\omega)$ of the holomorphic tangent bundle $TX^{1,0}$ of $X$ is a holomorphic section, so is a holomorphic vector field. We further define the space of \emph{holomorphy potentials} on $X$ to be the functions $h$ such that $\D h =  0,$ and throughout we denote $$\overline \mfk = \{ h \in C^{\infty}(X) \ | \ \D h = 0\},$$  so that for $\omega$ to be extremal means that $S(\omega) \in \overline \mfk$. The vector fields taking the form $\D h$ for some $h$ are precisely the holomorphic vector fields on $X$ that vanish somewhere.

Denoting by $\Aut_0(X)$ the connected component of the identity in the biholomorphism group of $X$, we further denote by $\Aut_0(X,\alpha) \subset \Aut_0(X)$ the Lie subgroup associated with vector fields that vanish somewhere. This is sometimes called the reduced automorphism group of $X$, and in the case that $\alpha = c_1(L)$ for some ample line bundle $L$ on $X$, corresponds to automorphisms which lift to $L$. Note that the group itself is actually independent of $\alpha$.

\subsection{Extremal metrics on blowups} 
As the main novelty in our work is the geometric approach, we give a detailed summary of the approach before establishing the various steps involved.

We consider a fixed compact complex manifold $X$, a K\"ahler class $\alpha$, and assume that there is an extremal K\"ahler metric $\omega\in\alpha$. We fix a point $p\in X$ and consider the blowup $\sigma: \Bl_p X\to X$ of $X$ at $p$, with exceptional divisor $E$. We wish to characterise the existence of extremal K\"ahler metrics on the blowup $\Bl_p X$ in the K\"ahler classes $\alpha_{\epsilon} = \sigma^*\alpha - \epsilon^2 [E]$ for $0<\epsilon\ll 1$. The argument consists of three steps. The first step is purely analytic and solves a general gluing problem, and no stability hypotheses enter into this step. The second  explains, having solved the gluing problem, how relative K-stability characterises the existence of extremal metrics on $(\Bl_p X, \alpha_{\epsilon})$. This already completely solves the existence problem for extremal K\"ahler metrics in these K\"ahler classes,  through relative K-stability. The third step then geometrically interprets what relative K-stability means in terms of more traditional geometric invariant theory, namely through geometric information around the point $p$ itself.

\subsubsection{Step 1: the main gluing argument} We begin with the case that $(X,\alpha)$ admits a cscK metric. The first main point of our argument is not to consider merely the blowup of $X$ at $p$, but instead to consider the blowup of $X$ at all points at once. That is, we consider the diagonal $$\Delta = \{(p,p) : p\in X\}  \subset X\times X,$$ which is a complex submanifold. We then consider the blowup $$\X = \Bl_{\Delta} (X\times X) \to X\times X$$ and let $\pi: \X  \to X$ be the projection onto the first factor. This holomorphic submersion is the universal blowup of $X$ along points; the fibre of $\pi$ over a point $p\in X$ is $\Bl_p X$. 

Letting $\E$ denote the exceptional divisor of the blowup $\sigma: \X \to X\times X$ (extending our previous notation for $\sigma$), this blowup comes with a natural $\epsilon$-dependent family of relatively K\"ahler classes $\scA_{\epsilon} = \sigma^*\alpha - \epsilon^2 [\E]$ for all $0<\epsilon \ll 1$, with $\alpha$ pulled back from one factor of $X$. We will consider in detail the geometry of the family of K\"ahler manifolds $\pi: (\X, \scA_{\epsilon}) \to X$.

The purely analytic first step of our argument constructs a relatively K\"ahler representatives of $\scA_{\epsilon}$ whose scalar curvature lies in a fixed obstruction space. The function space, which we denote $\bar \mfk_{V,\epsilon}$, is defined as a space of fibrewise holomorphy potentials with respect to an initial sequence of relatively K\"ahler metrics. For a relatively K\"ahler metric $\omega_{\epsilon} \in \scA_{\epsilon}$, denote by  $\omega_{\epsilon,p}$ the restriction of $\omega_{\epsilon}$ to $\Bl_pX$ and denote further $S_V(\omega_{\epsilon})$ the vertical scalar curvature, namely the function whose restriction to a fibre $\Bl_p X$ is the scalar curvature of $\omega_{\epsilon,p}$.

\begin{theorem}\label{sec2:gluing}
There is a sequence of relatively K\"ahler metrics $\omega_{\epsilon} \in \scA_{\epsilon}$ with  $$S_V(\omega_{\epsilon}) \in \bar \mfk_{V,\epsilon}.$$ \end{theorem}

 The notation $\omega_{\epsilon}$ is justified by the fact that these relatively K\"ahler metrics converge as $\epsilon \to 0$ to the pullback of $\omega$ to $X$. We point out here that the function space $\mfk_{V,\epsilon}$ is defined in such a way that, in proving this result, we may ultimately apply a version of the implicit function theorem fibrewise, rather than globally on $\X$. 

This result can be thought of as performing a version of the Arezzo--Pacard--Singer theorem \cite{APS} in families, where the the main point is to perform the relevant analysis in such a way that the resulting relatively K\"ahler metric varies smoothly from fibre to fibre, and to prove this we adapt the arguments of Seyyedali--Sz\'ekelyhidi (whose approach is more ``coordinate-free'') \cite{seyyedali-szekelyhidi}. The key point is that the function space $\bar \mfk_{V,\epsilon}$ is different to that involved in prior work, and this discrepancy will allow us in the subsequent step to employ stronger geometric results from K\"ahler geometry.

\subsubsection{Step 2: moment map geometry} In the next step, we wish to understand the geometry of the holomorphic submersion $(\X,\omega_{\epsilon}) \to X$. As we are interested for the moment in the cscK setting, we assume that $(\Bl_pX,\alpha_{\epsilon})$ is K-stable for $0<\epsilon \ll 1$, and aim to construct cscK metrics in the class $\alpha_{\epsilon}$.

We use the relatively K\"ahler metric $\omega_{\epsilon}$ to define a Hermitian metric on the relative anticanonical class $-K_{\X/X}$, with curvature which we denote $\rho_{\epsilon} \in c_1(-K_{\X/X}).$ Denote by $$\hat S_{V,\epsilon} = \left(\frac{-n\int_{\Bl_pX}c_1(\Bl_pX)\cdot\alpha_{\epsilon}^{n-1}}{\int_{\Bl_pX}\alpha_{\epsilon}^{n}}\right)$$ the vertical average scalar curvature, where we note that this value is independent of $p\in X$. We then consider the forms $\Omega_{\epsilon}$ on $X$ defined as fibre integrals $$\Omega_{\epsilon} =\frac{\hat S_{V,\epsilon}}{n+1} \int_{\X/X} \omega_{\epsilon}^{n+1} - \int_{\X/X}\rho_{\epsilon}\wedge \omega_{\epsilon}^n,$$ which is a closed $(1,1)$-form that is (by definition) the Weil--Petersson form on the base of the submersion $\X\to X$.

The group $\Aut_0(X,\alpha)$ is reductive \cite[Proposition 4.18]{gabor-book}, and we let $K$ be a maximal compact subgroup. The diagonal $\Delta \subset X\times X$ is fixed by $K$, so that there is a $K$-action on $\X$, and by equivariance of $\pi: \X\to X$ this implies that $\Omega_{\epsilon}$ is $K$-invariant. We then define a function space $\bar\mfk_{\pi,\epsilon}$ and moment maps $$\mu_{\epsilon}: \X \to \mfk^*$$ with respect to $\omega_{\epsilon}$, where $\mfk = \Lie K$ and where we interpret the moment map condition through equivariant differential geometry (noting $\omega_{\epsilon}$ may not be positive), such that for all $v\in \mfk$ we have $\langle\mu_{\epsilon}, v\rangle \in \bar\mfk_{\pi,\epsilon}$.  The moment map $$\sigma_{\epsilon}: X\to \mfk^*$$ for the $K$-action on $(X,\Omega_{\epsilon})$ then takes the form  \cite{DH}  $$\langle\sigma_{\epsilon},v\rangle = \int_{\X/X}\langle\mu_{\epsilon}, v\rangle(S_V(\omega_{\epsilon})- \hat S_{V,\epsilon}) \omega_{\epsilon}^n,$$ which as we are integrating an $(n,n)$-form, produces a function on $X$.

For $p\in X$ to be a zero of of the moment map $\sigma_{\epsilon}$ thus means  that on $\Bl_pX$ the scalar curvature $S(\eta_{\epsilon,p})-\hat S_{V,\epsilon}$ is $L^2$-orthogonal to $\bar\mfk_{\pi,\epsilon}|_{\Bl_pX}.$ So if we knew that the scalar curvature satisfied  $$S_V(\omega_{\epsilon}) - \hat S_{V,\epsilon} \in \bar \mfk_{\pi,\epsilon},$$zeroes of the moment map would be cscK metrics. It is important to emphasise here that the function space $\bar\mfk_{\pi,\epsilon}$---which is defined globally on $\X$---does not actually agree with $\bar\mfk_{V,\epsilon}$, where our relatively K\"ahler metrics $\omega_{\epsilon}$ satisfy $S_V(\omega_{\epsilon}) \in \bar\mfk_{V,\epsilon}$. 

Importantly,  however, the two function spaces do agree to leading order in $\epsilon$. The first crucial point is that this equality-to-leading-order is sufficient to reduce finding cscK metrics to finding  zeroes of the moment maps $\sigma_{\epsilon}$. The second crucial point is that although we do not know that the Weil--Petersson form $\Omega_{\epsilon}$ on $X$ is actually positive, we reduce the question to some in principle different---but genuinely K\"ahler---form on $X$ with the same moment map; this uses a more refined understanding of the expansion of the scalar curvature in $\epsilon$ in Theorem \ref{sec2:gluing}. Thus the $\sigma_{\epsilon}$ are moment maps with respect to a sequence of K\"ahler metrics on $X$, and to solve the problem we must only solve a finite-dimensional moment map problem on $X$ itself.

The general theory surrounding the Kempf--Ness theorem then explains precisely when such moment map problems may be solved. The output of these results is that if no zero of the moment map exists in the given $\Aut_0(X,\alpha)$-orbit of $p\in X$, there is an element $v \in \mfk$ such that the flow of the vector field $Jv$ (with $J$ the almost complex structure on $X$) specialises $p$ to some fixed point $q \in X$ of $v$, and such that $$\langle\sigma_{\epsilon},v\rangle(q) \leq 0,$$ with equality only when $q \in \Aut_0(X,\alpha)\cdot p$. Furthermore, $v$ can be taken to be rational (hence generating a $\C^*$-action). By definition of the moment map $\sigma_{\epsilon}$, the quantity $\langle\sigma_{\epsilon},v\rangle(q)$  is precisely the classical Futaki invariant computed on the blowup $(\Bl_{q}X,\alpha_{\epsilon})$. The element  $v \in \mfk$ produces a test configuration involved in the theory of K-stability, and the condition of  K-stability of $(\Bl_pX,\alpha_{\epsilon})$ then precisely implies that no such element $v$ may exist. Thus through the geometric approach we have taken, we end up in a situation where K-stability can directly be seen as the obstruction to the existence of cscK metrics; more precisely, we obtain that K-stability of $(\Bl_pX,\alpha_{\epsilon})$ implies the existence of a cscK metric on the blowup of $X$ at $p$ in each $\alpha_{\epsilon}$ for $0<\epsilon \ll 1$:

\begin{theorem} There is an $\epsilon_0>0$ such that if $0<\epsilon<\epsilon_0$ and  $(\Bl_p X, \alpha_{\epsilon})$ is K-stable, then $\alpha_{\epsilon}$ admits a cscK metric.
\end{theorem}

The converse, that the existence of a cscK metric implies K-stability, follows from  existing general theory \cite{stoppa, stoppa-szekelyhidi, relative, kahler,zak}. Thus we obtain the version of the Yau--Tian--Donaldson conjecture in this setting, with the additional information that only test configurations induced by one-parameter subgroups of $ \Aut_0(X,\alpha)$ are needed to test K-stability (in particular with smooth central fibre given by the blowup of $X$ at the specialisation of $p$).

We next turn to the extremal case, where our assumption is that $(\Bl_pX,\alpha_{\epsilon})$ is relatively K-stable for $0<\epsilon \ll 1$ and we wish to construct extremal metrics. We thus fix $p \in X$ and a maximal compact torus $T\subset \Aut(X,\alpha)_p$ (with the latter denoting the stabiliser of $p\in X$), and consider the action of the centraliser $\Aut(X,\alpha)^{T}$ of $T$ in $\Aut(X,\alpha)$ on the fixed point locus $X^{T}$, which is a smooth compact complex submanifold of $X$ and which contains $p$. We consider the restriction of our previous procedure to $X^{T}$, namely blowing up the diagonal $\Delta^T\subset X^T\times X^T$, giving $\X^T\to X^T$ which is an $\Aut(X,\alpha)^{T}$-equivariant holomorphic submersion with a trivial $T^{\C}$-action on both $\X^T$ and $X^T$. 

Lying in the Lie algebra $\Lie T^{\C}$ are canonically defined $\epsilon$-dependent vector fields $\xi_{\epsilon}$, which are (by definition) the extremal vector field on $(\Bl_pX,\alpha_{\epsilon})$ and which are independent of point in the orbit  $\Aut_0(X,\alpha)^T.p$. In the cscK case already considered, the $\xi_{\epsilon}$ vanish.  We obtain a natural sequence of closed $(1,1)$-forms $$\Omega_{\epsilon} =\frac{\hat S_{V,\epsilon}}{n+1} \int_{\X/X}\langle\mu_{\epsilon},\xi_{\epsilon}\rangle \omega_{\epsilon}^{n+1} - \int_{\X/X}\rho_{\epsilon}\wedge \omega_{\epsilon}^n,$$ which by a similar trick to before we may assume are K\"ahler, along with moment maps  $\sigma_{\epsilon}: X \to (\mfk^{T})^*$ taking the form (by a result of Hallam) $$\langle\sigma_{\epsilon},v\rangle = \int_{\X/X}\langle\mu_{\epsilon}, v\rangle(S_V(\omega_{\epsilon})- \hat S_{V,\epsilon} - \langle \mu_{\epsilon},\xi_{\epsilon}\rangle) \omega_{\epsilon}^n.$$ The rest of the argument is similar to the cscK case, once this geometry has been setup: the relative K-stability condition forces the existence of a zero of these moment maps, and by the gluing step, the existence of extremal metrics much as before.

\begin{theorem}There is an $\epsilon_0>0$ such that if $0<\epsilon<\epsilon_0$ and  $(\Bl_p X, \alpha_{\epsilon})$ is relatively K-stable, then $\alpha_{\epsilon}$ admits an extremal metric.
\end{theorem}

\subsubsection{Step 3: K-stability and GIT} What remains is to give an explicit understanding of relative K-stability of $(\Bl_{p}X,\alpha_{\epsilon})$ in terms of the geometry of the point $p\in X$ itself. Results of this form go back to Stoppa, Stoppa--Sz\'ekelyhidi and Sz\'ekelyhidi \cite{stoppa,stoppa-szekelyhidi,gabor-blowups2}, and in essence show that the Futaki invariant $$\langle\sigma_{\epsilon},u\rangle(q)  = \int_{\Bl_{q} X} \mu_{\epsilon}(v)(S(\omega_{\epsilon,q}) - \hat S_{V,\epsilon}) \omega_{\epsilon,q}^n$$ has a complete asymptotic expansion in $\epsilon$ involving the geometric invariant theory (GIT) weight of the point $q$ with respect to $v$ and the classes $\alpha$ and $c_1(X)$. Our contribution in this step is to extend these results also to the inner products involved in the definition of relative K-stability, which then allows us to give an explicit equivalent criterion for relative K-stability in terms of a relative version of GIT stability.

Thus we show that relative K-stability, with respect to the test configurations induced by Step 2, is equivalent to a suitable relative version of GIT stability. The Kempf--Ness theorem applied again equivalently characterises relative GIT stability in terms of a sequence of finite-dimensional moment maps. So we see that both relative GIT stability---and equivalently the existence of zeroes of corresponding finite-dimensional moment maps---also characterise  relative K-stability and hence the existence of extremal metrics on $(\Bl_{p}X,\alpha_{\epsilon})$. This completes the proof in the case that $(X,\alpha)$ admits an extremal metric.

\subsection{The semistable case} We next turn to the semistable case, which again splits into a cscK version and an extremal version. To appeal to a result requiring a kind of compactness, we will assume $\alpha = c_1(L)$ for $L$ ample, so that $X$ is projective. As the main novelty is the same in both cscK and extremal settings, we explain only the cscK setting, where our assumption is that $(X,L)$ is analytically K-semistable. By definition, this means that $(X,L)$ is a small deformation of a cscK  manifold, which we write $(\Y_{0},\L_0)$. We consider the Kuranishi space $B$ of $(\Y_{0},\L_{0})$, which we assume is smooth (namely we assume that the deformation theory of $(\Y_{0},\L_0)$ is unobstructed).  By construction, $B$ admits an action of a maximal compact subgroup $K$ of $\Aut_0(\Y_{0},\L_0)$, and there is a universal family $\pi_{\Y}: (\Y,\L) \to B$ with a $K$-action making $\pi_{\Y}$ equivariant. We will be interested in a point $p \in X$ and will assume that a maximal torus in $\Aut_0(\Y_{0},\L_0)_p$ induces a maximal torus of $\Aut_0(X,\alpha)_p$, with $(X,\alpha)$ viewed as a fibre in the Kuranishi space (this holds automatically, for example, if $p$ has trivial stabiliser).

The fibre product $\Y\times_B \Y$ is a complex manifold which, as a set, consists of pairs $(y_1,y_2)$ with $y_1$ and $y_2$ lying in the same fibre over $B$, and which admits a holomorphic submersion over $\Y$ with fibre over $y\in \Y$ given by $\Y_{\pi(y)}$ (the fibre of $\pi_{\Y}$ over $\pi_{\Y}(y)$). We consider the diagonal $$\Delta = \{(y,y) \in \Y\times_B \Y \} \subset \Y\times_B \Y,$$ which is smooth (being biholomorphic to $\Y$) and its blowup $$\X = \Bl_{\Delta} (\Y\times_B \Y).$$ We thus obtain a holomorphic submersion $\sigma: \X \to \Y$ with fibre over $y\in \Y$ given by $\Bl_y \Y_{\pi(y)}$, along with classes $\L_{\epsilon} = \sigma^*\L - \epsilon^2[\E]$, which are relatively ample over $\Y$. 

We are thus in a similar situation to the cscK case. The main difference is that the base of the submersion $\Y$ is not compact, although the morphism $\Y \to B$ is proper. A similar procedure to the cscK case endows the class $\L_{\epsilon}$ with a sequence of relatively K\"ahler metrics $\omega_{\epsilon}$, such that on each fibre  $(\Bl_y\Y_{\sigma(y)},\L_{\epsilon})$ their scalar curvature lies in a function space $\bar\mfk^T_{V,\epsilon}|_{\Bl_y\Y_{\sigma(y)}}$, perhaps after shrinking $B$ (which is necessitated by its noncompactness). 

We then endow the base $\scY$ of the submersion with a sequence of K\"ahler metrics, by producing a sequence $\Omega_{\epsilon}$ of closed $(1,1)$-forms defined as fibre integrals over the submersion $\X \to \Y$ through the relatively K\"ahler metrics $\omega_{\epsilon}$ in the same way as before (involving also the induced metric on the relative antitcanonical class $-K_{\X/\scY}$), and using a similar trick as in the cscK case we may assume that the $\Omega_{\epsilon}$ are actually K\"ahler, perhaps after shrinking $B$. Since $(X,L)$ does not admit a cscK metric, the associated point $\pi(y)\in B$ is strictly semistable in the sense of GIT, and the main point is that the point $y\in \Y$ may nevertheless be stable with respect to the $\Omega_{\epsilon}$. So our main point is to understand the geometry of the  natural sequence of moment maps $\sigma_{\epsilon}$ for these K\"ahler metrics $\Omega_{\epsilon}$. We prove the following:

\begin{theorem} There is an $\epsilon_0>0$ such that if $0<\epsilon<\epsilon_0$ and  $(\Bl_p X, L_{\epsilon})$ is K-stable, then $c_1(L_{\epsilon})$ admits a cscK metric.\end{theorem}

General theory again provides a converse. The main difference with the cscK and extremal settings is that the space  $\Y$---which is the base of our holomorphic submersion $\X\to \Y$---is non-compact, meaning we cannot appeal to the global Kempf--Ness theorem as before to relate GIT stability to the geometry of the  moment map. We thus use a generalisation of the ``gradient flow'' approach introduced in \cite{DMS}, which is more suited to the noncompact setting. What is ultimately needed is that the flow exists for all time; for this, and also to understand the asymptotics of the flow, we must embed $\Y$ in a compact space, where some new arguments are required, involving the projectivity hypothesis on $(X,L)$.

\section{The main gluing argument}
\label{sec:analysis}
The goal of this section is to prove Theorem \ref{sec2:gluing}.

\subsection{Constructing an initial approximate solution}
Let $\X = \Bl_{\Delta}(X \times X)$, where $\Delta \subset X \times X$ is the diagonal. We will let $\sigma : \X \to X \times X$ be the blowdown map, $\tau : X \times X \to X$ be the projection to the first factor, and $\pi = \tau \circ \sigma.$ We therefore have the following diagram:
$$
\begin{tikzcd}
  \X \arrow[r, "\sigma"] \arrow[rd, "\pi"]
    & X \times X \arrow[d, "\tau"] \\
	& X 
\end{tikzcd}
$$

Let $\omega$ be a fixed metric on $X$.  We will take $\omega$ to be an extremal metric on $X$, but this is not important for the moment. This gives a relative K\"ahler metric, which we also denote $\omega$, on $X \times X \to X$ by pulling back from the projection to the \emph{second} factor. Next, let $d$ be the function on $X \times X$ whose restriction to a fibre of $\phi$ is the distance with respect to $\omega$ to the diagonal, within that fibre. In other words, $d$ is the function whose restriction to $\tau^{-1}(p)$ is the distance to $p$. Note that as $\Delta \subset X \times X$ is a smooth subvariety, and $d^2$ is smooth on every fibre of $\phi$, $d^2$ is smooth on $X \times X$.

We will fix a maximal torus $T$ of the stabiliser $\Aut_0(X,\alpha)_p$ of $p$. Since $\omega$ is in particular invariant under the action of $T$,  the distance function $d$ is invariant under the action of $T$.

The Burns--Simanca metric $\eta_{\BS}$ is a scalar-flat metric on $\Bl_0 \C^n$ \cite[Section 8.1.2]{gabor-book}. Let $\gamma$ be a cutoff function which is equal to $1$ on $[-\infty,1]$ and $0$ on $[2, \infty)$. We can identify the complement of the exceptional divisor in $\Bl_0 \C^n$ with $\C^n \setminus \{ 0 \}$ which we give coordinates $\zeta$. The Burns-Simanca metric can be written as
$$
\eta_{\BS} = \ddb \left( |\zeta|^2 + \gamma(|\zeta|) \log(|\zeta|^2)+ f(|\zeta|^2) \right),
$$
where $f : [0, \infty)$ is smooth up to the boundary of $[0, \infty)$, and such that $\nabla^i f(s) = O(s^{2-n-i})$ as $s \to \infty$, for all $i$. Now, let $r_{\epsilon} = \epsilon^{\frac{2n}{2n+1}}$, and let $\gamma_2 = \gamma(r_{\epsilon}^{-1} d)$ and $\gamma_1=1-\gamma_2$. Define
$$
\omega_{\epsilon} = \sigma^*\omega + \epsilon^2 \ddb\left(\gamma_2 \cdot (\gamma(\epsilon^{-1}d)\log (\epsilon^{-2}d^2) + f(\epsilon^{-2}d^2))\right).
$$ 
\begin{remark}
In \cite{seyyedali-szekelyhidi}, Seyyedali--Sz\'ekelyhidi study the extremal metric equation on the total space of blowing up an extremal K\"ahler manifold in a submanifold. They define a K\"ahler metric on the total space using a similar formula to the above. The main difference, in addition to starting with a K\"ahler metric instead of a relatively K\"ahler metric before blowing up, is that they use the distance function to the submanifold, whereas we use the fibrewise distance function. 

Note also that in the case of blowing up a point, the forms defined by Seyyedali--Sz\'ekelyhidi differ slightly from those defined by Sz\'ekelyhidi \cite{gabor-blowups1}. Sz\'ekelyhidi glues $\omega$ and $\epsilon^2 \eta_{\BS}$, where $\eta_{\BS}$ is the Burns--Simanca metric,  over an annular region, whereas Seyyedali--Sz\'ekelyhidi add $\epsilon^2 \eta_{\BS}$ to $\pi^* \omega$ near the exceptional divisor, and then cut off the Burns--Simanca metric over the annular region.
\end{remark}

\begin{lemma}
The closed two-form $\omega_{\epsilon}$ is relatively K\"ahler with respect to $\pi : \X \to X$, for all sufficiently small $\epsilon >0$.
\end{lemma}
\begin{proof}
The $\omega_{\epsilon}$ restrict to the forms defined by Seyyedali--Sz\'ekelyhidi \cite{seyyedali-szekelyhidi} on each fibre, which are K\"ahler for all sufficiently small $\epsilon$ (depending on the fibre). We musst show that we can make this uniform in the point $p$, which, by compactness of the base of $\X \to X$, boils down to showing that for every $p \in X$ we can find such a uniform estimate in a ball about $p$. 

Indeed, given a holomorphic normal coordinate system about $p$, one can choose $\epsilon_0>0$ depending on $\omega$ and the size of this coordinate system so that $\omega_{\epsilon}$ is K\"ahler for all $\epsilon \in (0, \epsilon_0)$, see \cite[Proposition 4 and Lemma 5]{seyyedali-szekelyhidi}. If the size of the holomorphic normal coordinate system is shrunk by a factor of $\lambda$, then the corresponding $\epsilon_0$ is also shrunk by a power of $\lambda$. Now, if we have such a coordinate system about $p$ which has size $r$, then we can ensure every point that is sufficiently close to $p$ has a holomorphic normal coordinate system about it of radius $\frac{r}{2}$. It then follows that 
there is a ball about $p$ such that for all $q$ in this ball, $\omega_{\epsilon}$ is K\"ahler for all $\epsilon \in (0, c \epsilon_0)$, where $c \in (0,1)$ is a constant. This is the required uniform local bound, giving the result.
\end{proof}

As the diagonal is fixed by the product action of $\Aut_0(X,\alpha)$ on $X\times X$, this action lifts to $\X$. In this way, holomorphic vector fields on $X$ induce holomorphic vector fields on $\X$. Our goal in this section is construct a relatively K\"ahler metric whose fibrewise scalar curvature is equal to the restriction of a sort-of holomorphy potential for such a vector field. We begin by describing how the holomorphy potentials for these vector fields lift to $\X$.

First note that if we used the form $\widetilde \omega_{\epsilon}$ which is produced by using the distance function $d_{\Delta}$ to the diagonal $\Delta$, computed with respect to the product metric of $\omega$ on both factors of $X\times X$, we would obtain a genuinely K\"ahler metric $\widetilde \omega_{\epsilon}$ on $\X$ for all sufficiently small $\epsilon$. On the region $B_{r_{\epsilon}}$ about the exceptional divisor, the metric in coordinates is a perturbation of  $\pi^* \omega + \epsilon^2 \eta$, where $\eta$ is the Burns-Simanca metric. It follows that if $\xi$ is a real holomorphic vector field on $X \times X$ with potential $h'$ (which is just the product of the potentials with respect to the factors), then the lift of $\xi$ (which we will also denote $\xi$) has potential $\widetilde h_{\epsilon}$ such that
$$
\widetilde h_{\epsilon} - h' = \gamma_2 \cdot O(1),
$$ 
since $\omega_{\epsilon} = \pi^* \omega$ on $B_{2r_{\epsilon}}$, $\widetilde h_{\epsilon} - h'$ is supported in $B_{2r_{\epsilon}}$. Here, by $\gamma_2 \cdot O(1)$ we mean that $|\widetilde h_{\epsilon} - h' |$ is bounded by a constant times $\gamma_2$.

Now, we can write
$$
\omega_{\epsilon} = \widetilde \omega_{\epsilon} - (\tau_2 \circ \pi)^* \omega + \ddb \phi_{\epsilon},
$$
where $\tau_2 : X \times X \to X$ is the projection to the second factor and $\phi_{\epsilon}$ is the difference of potentials, given by
\begin{align*}
\phi_{\epsilon} =& \epsilon^2 \left(\gamma_2 \cdot (\gamma(\epsilon^{-1}d)\log (\epsilon^{-2}d^2) + f(\epsilon^{-2}d^2))\right) \\
-& \epsilon^2\left(\gamma_2 \cdot (\gamma(\epsilon^{-1}d_{\Delta} )\log (\epsilon^{-2}d_{\Delta}^2) + f(\epsilon^{-2}d_{\Delta}^2))\right) .
\end{align*}
If $h$ denotes the holomorphy potential of some  $u\in \mfk^T$ on  $X$, then the holomorphy potential with respect to $(\tau_2 \circ \sigma)^* \omega$ is the pullback of $h$. 
It follows that 
\begin{align}
\label{eq:pots}
h_{\epsilon} = \widetilde h_{\epsilon} - (\tau_2 \circ \sigma)^*h + \frac{1}{2} u (f_{\epsilon})
\end{align}
is a holomorphy potential with respect to $\omega_{\epsilon}$. Now, $u (f_{\epsilon})$ is $O(1)$ in $\epsilon$ as we are taking a derivative in the above. However, the support of $f_{\epsilon}$ is also contained in the support of $\gamma_2$ as this is a factor in $f_{\epsilon}$, which is contained in $B_{2r_{\epsilon}}$. Thus
$$
h_{\epsilon} = h + \gamma_2 \cdot O(1),
$$ 
where $h$ is identified with its pullback from the first factor. 

\subsection{Weighted H\"older spaces}
\label{sec:wtdspaces}

We now define the weighted H\"older spaces relevant to our problem, following \cite{arezzo-pacard1,gabor-blowups1}.

\subsubsection{Weighted spaces on punctured manifolds}

We now define weighted spaces on the fibres. Fix holomorphic coordinates $z_1, \ldots, z_n$ around $p$ in which $T$ acts linearly and which are normal with respect to the K\"ahler metric induced by $\omega$, and assume that they exist at least in the disk $D_2$ of radius $2$ about $p$; this can be assumed after scaling $\omega$ if necessary. 

Given a function $f : M \setminus \{ p \} \to \R$, define $f^{\delta}_r : D_2 \setminus D_1 \to \mathbb{R}$ for  for $r>0$ by 
$$
f^{\delta}_r (z) = r^{-\delta} f(rz).
$$

\begin{definition}
\label{def:wtdnormpuncture}
The $C^{k,\alpha}_{\delta}$-\emph{norm} on $M \setminus \{ p \}$ is defined to be 
$$
\| f \|_{C^{k, \alpha}_{\delta} (M \setminus \{ p \} )} = \textnormal{sup}_{r \in (0,1)} \| f_r^{\delta} \|_{C^{k,\alpha}} ( D_2 \setminus D_1) + \| f \|_{M \setminus D_1 (p) }.
$$
\end{definition}

\subsubsection{Weighted spaces on the blowup}

We have coordinates $z$ about $p$ which we will identify with coordinates $\zeta = \epsilon^{-1} z$ about the exceptional divisor in $\Bl_0 \C^n$. We identify the annulus $D_1 \setminus D_{\epsilon}$ around $p$ in $X$ with the (preimage via the blowdown map of the) annulus $D_{\epsilon^{-1}} \setminus D_1$ in $\Bl_0 \mathbb{C}^n$. In other words, 
$$
\Bl_p X = X \setminus D_{\epsilon}(p)   \bigcup_{D_1 \setminus D_{\epsilon} \sim \pi^{-1}( D_{\epsilon^{-1}} \setminus D_1 )} \pi^{-1} (D_{\epsilon^{-1}}) .
$$
We can then also define a weighted norm on the blowup. Given a function $f : \Bl_p X \to \mathbb{R}$, we can define a function $f^{\delta}_{\varepsilon} : \Bl_0 D_1 \to \mathbb{R}$ by
$$
f^{\delta}_{\varepsilon} (\zeta) = \varepsilon^{-\delta} f (\varepsilon \zeta).
$$
Up to a rescaling depending on $\varepsilon$, this is the restriction of $f$ to the preimage via the blowdown map of the ball of radius $\varepsilon$ about $p$, pulled back to the preimage of a ball of fixed size.

\begin{definition}
\label{def:wtdnorm}
The $C^{k,\alpha}_{\delta}$-\emph{norm} on $\Bl_p X$ is defined to be 
$$
\| f \|_{C^{k, \alpha}_{\delta}(\Bl_p X) } = \| f_{\epsilon}^{\delta} \|_{C^{k,\alpha} ( \Bl_0 D_1 )} +  \textnormal{sup}_{r \in (\varepsilon,1)} \| f_r^{\delta} \|_{C^{k,\alpha}} (D_2 \setminus D_1) + \| f \|_{X \setminus D_1 (p) }.$$
\end{definition}

\subsection{The linearisation}
The next step is to understand the linearisation of our problem. The Lichnerowicz operator, which approximates the linearisation of scalar curvature operator, has (co)-kernel which depends on the blown up point. We assume that the maximal torus $T\subset \Aut_0(X,\alpha)_p$ of $\Aut_0(X,\alpha)_p$ contains the extremal vector field, so that $S(\omega) \in \mft = \Lie T\subset \mfk^T$, and note  that $X^T$ is a closed submanifold of $X$. 

We begin by recalling the linearisation result on the punctured manifold $X_p = X \setminus \{ p \}$. Let $T$ be a torus of isometries of $(X,\omega)$ and let $(C^{k, \alpha}_{\delta})^T$ denote the space of $T$-invariant functions, while $\bar \mfk$ denotes the space of $T$-invariant Hamiltonian functions of holomorphic killing fields.
\begin{proposition}[{\cite[Proposition 17]{gabor-blowups1}}]
\label{prop:Lichpunctured}
Let $\delta <0.$ Then, for all $p \in X$, the operator 
$$
P : (C_{\delta}^{4,\alpha})^T(X_p) \times \bar \mfk^T \to (C_{\delta-4}^{0,\alpha})^T (X_p)
$$
given by
$$
(f,\nu,c) \mapsto \mathcal{D}_{\omega}^* \mathcal{D}_{\omega} (f)  - h
$$
admits a right-inverse $Q_{p}$ with operator norm bounded independently of $\epsilon$. Moreover, $Q_{p}$ depends smoothly on $p$.
\end{proposition}
The only part that is new compared to \cite{gabor-blowups1} is the smooth dependence on $p$, which follows by taking the $C^{4,\alpha}$-component of $Q_{\epsilon,p}$ to be orthogonal to $\bar \mfk^T$.

A similar statement holds for the Lichnerowicz operator on $\Bl_0 \C^n$ associated to the Burns--Simanca metric $\eta$.
\begin{proposition}[{\cite[Proposition 18]{gabor-blowups1}}]
\label{prop:LichBS}
Suppose $n>2$ and let $\delta > 4-2n.$ Then, the operator 
$$
P : C_{\delta}^{4,\alpha}(\Bl_0 \C^n) \to C_{\delta-4}^{0,\alpha} (\Bl_0 \C^n)
$$
given by
$$
f\mapsto \mathcal{D}_{\eta}^* \mathcal{D}_{\eta} (f)  
$$
admits a right-inverse $Q$ with bounded operator norm. 

In the case when $n=2$, the same holds for $\delta \in (-1,0)$ for the operator
$$
P : C_{\delta}^{4,\alpha}(\Bl_0 \C^2) \times \R \to C_{\delta-4}^{0,\alpha} (\Bl_0 \C^2)
$$
given by
$$
(f,c) \mapsto \mathcal{D}_{\eta}^* \mathcal{D}_{\eta} (f) - c \chi,
$$
where $\chi$ is some fixed compactly supported function on $\Bl_0 \C^2$ with non-zero integral.
\end{proposition}

With this in place, we can prove the global result we need on the fibre $\Bl_pX$ of $\X$ over $p$. Let $L_{\omega_{\epsilon,p}}$ denote the linearisation of the scalar curvature operator at $\omega_{\epsilon,p}$. This can then be written 
$$
L_{\omega_{\epsilon,p}} (f) = - \mathcal{D}_{\epsilon,p}^* \mathcal{D}_{\epsilon,p} (f) + \frac{1}{2} \langle \nabla S(\omega_{\epsilon}), \nabla f \rangle.
$$ 
This is an operator on the fibre, so in the above we mean the pairing and gradient with respect to the metric induced by $\omega_{\epsilon}$ on the fibre.

We will et $\xi$ denote the lift of the extremal vector field (induced by $S(\omega) \in \bar \mfk$) to $\X$.  For $p \in X^T$ a fixed point of the torus $T$, the action of $T$ lifts to $\Bl_p X$, and is given by the restriction of the $T$-action on $\X$ to the fibre $\X_p$, which is fixed by $T$. We will let $C^{4,\alpha}_{\delta}(\Bl_pX)^T$ denote the subspace of $C^{4,\alpha}_{\delta}(\Bl_pX)$ of $T$-invariant functions. Since $\mfk^T$ is the Lie algebra of the commutator of $T$, all functions $h$ in $\bar \mfk^T$ and their lifts $h_{\epsilon}$ to $\X$ restricted to $\Bl_pX$ are $T$-invariant.
\begin{proposition}
\label{prop:linearisationinverse}
Let $n>2$ and let $\delta \in (4-2n, 0).$ Then, for all $p \in X^T$, the operator 
$$
P : C^{4,\alpha}_{\delta}(\Bl_pX)^T \times \bar \mfk^T \to C_{\delta-4}^{0,\alpha} (\Bl_pX)^T
$$
given by
$$
(f,h) \mapsto L_{\omega_{\epsilon}}(f) - \frac{1}{2} \langle \xi, \nabla f \rangle - h_{\epsilon} 
$$
admits a right-inverse $Q$ with operator norm bounded independently of $\epsilon$. Moreover, $Q$ depends smoothly on $p$.

In the case when $n=2$, the same holds for $\delta \in (-1,0)$ for all $|\delta|$ sufficiently small, but with $\| Q \| \leq C \epsilon^{\delta}$.
\end{proposition}
\begin{proof}
We explain how to adapt the proof of \cite[Proposition 22]{gabor-blowups1}. The proof first builds an approximate inverse to $P$, which is then perturbed to a genuine inverse. This approximate inverse is built from the inverses $Q_1$ of Proposition \ref{prop:Lichpunctured} to the Lichnerowicz operator of the extremal metric on the punctured manifold $X \setminus \{ p \}$ and $Q_2$  of Proposition \ref{prop:LichBS} to the Lichnerowicz operator of the Burns-Simanca metric on $\Bl_0 \C^n$, respectively.  As the domain of $P$ is $C^{4,\alpha}_{\delta}(\Bl_pX)^T \times \bar \mfk^T$, the approximate inverse maps to this product space and we describe the two components of the approximate inverse in turn. 

The $C^{4,\alpha}_{\delta}$-component $\widetilde{Q}$ of the approximate inverse of $P$ at $\phi \in C^{0,\alpha}_{\delta}$ can be written
$$
\widetilde{Q} (\phi) = \beta_1 Q_1 (\gamma_1 \phi) + \beta_2 Q_2(\gamma_2 \phi),
$$
where $\gamma_1 \phi$ is thought of as a function on $X \setminus \{ p \}$, $\beta_1$ is a cut-off function depending on $\epsilon$ which vanishes near $p$ and is equal to $1$ on the support of $\gamma_1$, and similarly for $\gamma_2 \phi$ on $\Bl_0 \C^n$. We refer to \cite{gabor-blowups1} for the precise definition of the $\beta_i$, but note that they do not add  to $1$ everywhere. 

The $\bar \mfk^T$-component of the approximate inverse is given by the $h_{\epsilon}$ (depending on $\phi$) associated to the $\bar \mfk^T$-component of $Q_1(\gamma_1 \phi)$ on $X \setminus \{ p \}$. That is, it is given by the lift $h_{\epsilon}$ of the function $h \in \bar \mfk^T$ that solves 
\begin{align}
\label{eq:Q1inverse}
\mathcal{D}^*_{\omega} \mathcal{D}_{\omega} Q_1(\gamma_1 \phi) - h = \gamma_1 \phi
\end{align}
on $X_p$. Since the inverse on $X_p$ is bounded, this satisfies that
\begin{align}
\label{eq:Q1bound}
\|Q_1(\gamma_1 \phi) \|_{C^{4,\alpha}_{\delta}(X_p)} + |h| \leq c \| \phi \|_{C^{0,\alpha}_{\delta-4}}
\end{align}

Now, to show that $\phi \mapsto (\widetilde{Q}(\phi), h_{\epsilon})$ is an approximate inverse to $P$, we need to show that the Claim in the proof of \cite[Proposition 22]{gabor-blowups1} holds, namely that
$$
\left\| L_{\omega_{\epsilon}}(\widetilde{Q} (\phi)) - \frac{1}{2} \xi (\widetilde{Q} (\phi) ) - h_{\epsilon} - \phi \right\|_{C^{0,\alpha}_{\delta-4}} \leq \frac{1}{2} \| \phi \|_{C^{0,\alpha}_{\delta-4}}
$$
for all sufficiently small $\epsilon$. The key difference between our setup and that of Sz\'ekelyhidi is that our lift $h_{\epsilon}$ of the Hamiltonian does not in general agree with his lift -- it only does so for holomorphy potentials for vector fields that lift to the blowup. The crucial difference on the part of the expression identified with a function on $X_p$ is therefore in the term
$$
\beta_1 h - \gamma_1 h_{\epsilon},
$$
which arises after manipulating the above terms and using Equation \eqref{eq:Q1inverse}. The bound for this follows in our case for the same reason as in \cite{gabor-blowups1}.  Indeed, since $h_{\epsilon} = h$, the holomorphy potential on $B_{\epsilon}$, outside $B_{2r_{\epsilon}}$, the term $\beta_1 h - \gamma_1 h_{\epsilon}$ on $X_p$ is supported in $B_{2r_{\epsilon}} \setminus \{ p \}$. Moreover, 
$$
\| \beta_1 h - \gamma_1 h_{\epsilon}\|_{C^{0, \alpha}_{0}} \leq c | h |.
$$
By the scaling property of the weighted norm on $X_p$ for functions supported in $B_{2r_{\epsilon}}$ and Equation \eqref{eq:Q1bound}, we therefore obtain that 
\begin{align*}
\| \beta_1 h_{\epsilon} - \gamma_1 h_{\epsilon}\|_{C^{0, \alpha}_{\delta - 4}} \leq c r_{\epsilon}^{4-\delta} \| \phi \|_{C^{0,\alpha}_{\delta-4}}.
\end{align*}

We also need to account for the change in the terms identified with a function on $\Bl_0 \C^n$, but this follows in a similar way. Here there is no term like $\beta_2 h_{\epsilon}$, since the Lichnerowicz operator of Proposition \ref{prop:LichBS} is invertible, not just invertible modulo $\mfk^T$. Thus one needs to bound $\gamma_2 h_{\epsilon}$. But this follows from the scaling property again, since the function is supported in the preimage of $B_{2r_{\epsilon}}$ via the blowdown map.

Note that by taking $Q_1$ to be the right inverse defined by taking the $C^{4,\alpha}_{\delta}$-component to be the unique one which is orthogonal to $\bar \mfk^T$, the operator $Q_1$ will depend smoothly on $p$. It follows that the right inverse $Q$ will do so as well.
The adjustments needed for the case $n=2$ follows from the mapping properties of the Lichnerowicz operator on $\Bl_0 \C^2$. They are therefore not impacted by the change in lift of the holomorphy potentials and follow exactly as in \cite{gabor-blowups1}.
\end{proof}

\subsection{Improving the approximate solution}
\label{sec:improving}
Next, we need to improve the approximate solutions $\omega_{\epsilon}$. We begin by changing $\omega$ on $X \times X \setminus \Delta$ similarly to \cite{seyyedali-szekelyhidi}. Let $\mathcal{D}^* \mathcal{D}$ be the Lichnerowicz operator associated to the product metric of $\omega$ on each factor on $X \times X$. Suppose first that $n>2$. If we let $\Gamma' = - d^{4-2n}$, then 
$$
\mathcal{D}^* \mathcal{D} (\Gamma') = - c \delta_{\Delta} + O(d^{5-2n}),
$$
for some constant $c$, where $\delta_{\Delta}$ is the current of integration along the diagonal $\Delta$. The same then holds fibrewise with respect to the relatively K\"ahler metric $\omega$ pulled back from the second factor, i.e. for every fibre $\phi^{-1}(p) \cong X$, we have that 
$$
\mathcal{D}_{\omega}^* \mathcal{D}_{\omega} (\Gamma') = - c \delta_{p} + O(d^{5-2n}).
$$
The constant is dimensional and does not depend on $p$. In the case $n=2$, one has a similar expansion when using $\log d$ instead.

From the mapping properties of Proposition \ref{prop:Lichpunctured} for the Lichnerowicz operator on the punctured manifolds $X_p$, it follows that there for every $p$, there is a 
$$
\Gamma_p = - d^{4-2n} + \widetilde \Gamma_p,
$$
where $\widetilde{\Gamma}_p$ is $O(d^{5-2n})$, and a $h_p \in \bar \mfk^T$ such that
$$
\mathcal{D}^* \mathcal{D} \Gamma_p = h_p - c \delta_{p}.
$$
Note that 
$$
\int_X h_p f \omega^n = c f(p)
$$
for every $f \in \bar\mfk^T$. Now, since $f$ is the hamiltonian for the vector field in $\mfk^T$ induced by the action of $K$ on $X$, we have that
$$
df (p) = d(\langle \mu, f \rangle) (p),
$$
where $\mu : X \to \mfk^*$ is the moment map. If we identify this with an element of $\bar \mfk^T$, then this is the element satisfying
$$
\langle \mu(p), f \rangle = \int_X \mu(p) f \omega^n
$$
for every $f$. Also, by the above
\begin{align*}
c df (p) = d\left(\int_X h_p f \omega^n\right),
\end{align*}
which means that $h_p$ equals $c$ times the moment map for the $K$-action, up to a constant. Integrating against the constant functions gives that 
$$
h_p = c \mu (p) + \frac{c}{\Vol(X)},
$$
where we have identified $\mu(p) \in \mfk^*$ with a function on $X$ via the $L^2$-pairing as above.

We now modify the relatively K\"ahler metric $\omega$ by the function $\Gamma$ whose restriction to $\phi^{-1}(p)$ is $\Gamma_p$. Note that this function is smooth on $X \times X \setminus \Delta$ since $d$ is and since by Proposition \ref{prop:Lichpunctured} the right-inverse of the Lichnerowicz operator on the fibres of $X \times X \to X$ depends smoothly on $p$ (so  $\widetilde \Gamma_p$ depends smoothly on $p$). We let
$$
\widetilde \omega_{\epsilon} = \omega_{\epsilon} + \ddb\left(\epsilon^{2n-2} \gamma_1  \Gamma \right).
$$

Letting $\omega_{\epsilon,p}$ denote the restriction of $\omega_{\epsilon}$ to $\Bl_pX$ and similarly for $\widetilde \omega_{\epsilon}$, we then have that 
\begin{align*}
S(\widetilde \omega_{\epsilon, p}) =& S(\omega_{\epsilon, p}) + L_{\omega_{\epsilon, p}} (\epsilon^{2n-2} \gamma_1  \Gamma) + R_{\omega_{\epsilon,p}}(\epsilon^{2n-2} \gamma_1  \Gamma)  \\
=&S(\omega_{\epsilon, p}) + L_{\omega_{\epsilon, p}} (\epsilon^{2n-2} \gamma_1  \Gamma) + O(\epsilon^{2n-1}),
\end{align*}
where $R_{\omega_{\epsilon,p}}$ is the non-linear part of the scalar curvature operator.

We begin by showing that this is an improved approximate solution. Let $h'_{\epsilon,p}$ denote the potential with respect to $\omega_{\epsilon}$ associated to $S(\omega) + \epsilon^{2n-2} h_p \in \bar \mfk$ and let $\xi'_{\epsilon,p}$ denote the corresponding real holomorphic vector field. 
\begin{lemma}[{\cite[Lemma 24]{gabor-blowups1}}]
\label{lem:approxsoln}
Let $\delta \in (4-2n, 0)$ be sufficiently close to $4-2n$ in the case $n>2$ and let $\delta < 0$ be sufficiently close to $0$ in the case $n=2$. Then, for all $0 < \epsilon \ll 1$ and for all $p \in X$,
\begin{align*}
\| S(\widetilde \omega_{\epsilon, p}) - \frac{1}{2} \xi'_{\epsilon,p}(\epsilon^{2n-2} \gamma_1  \Gamma_p) - h'_{\epsilon, p} \|_{C^{0,\alpha}_{\delta-4}} \leq C r_{\epsilon}^{4-\delta}.
\end{align*}
\end{lemma}
The proof is largely as in \cite{gabor-blowups1}. The main issue is to bound $S(\widetilde \omega_{\epsilon, p})$ on the annular region $B_{2r_{\epsilon}} \setminus B_{r_{\epsilon}}.$ The Burns--Simanca metric admits an expansion 
$$
\eta = \ddb \left( \frac{d^2}{2} - d^{4-2n} + \tilde \psi \right),
$$ 
where $\psi$ is $O(d^{3-2n}).$ 
Since $\omega$ now has this expansion as well, this means that the subleading order term in the expansion of the potential of $\widetilde \omega_{\epsilon,p}$ is $\epsilon^{2n-2} d^{4-2n}$ -- without any cutoff function. The crucial point is that $d^{4-2n}$ agrees with $|z|^{4-2n}$ to leading order in coordinates which, with respect to the flat metric, is in the kernel of the Laplacian squared, the linearised operator at the flat metric. Thus this term has the $C^{0,\alpha}_{\delta-4}$-norm of an $O(d^{3-2n})$-function, which has the required upper bound of the form $C r_{\epsilon}^{4-\delta}$ on the annular region.

\subsection{Solving the non-linear equation fibrewise}
We are now ready to prove Theorem \ref{sec2:gluing}, which we will state more  precisely as Theorem \ref{thm:fibrewisesoln} below. We first define the space $\bar \mfk_{V,\epsilon}$. Let $h_{u, \epsilon}$ denotes the holomorphy potential of the lift $u$ of a holomorphic vector field from $X \times X$ to $\X$ with respect to $\omega_{\epsilon}$. Let $h_{u, \epsilon,p}$ denote the restriction of such a function $h_{\epsilon}$ to a fibre $\X_p=\Bl_pX$. For a function $\phi$ on $\X_p$ such that $\omega_{\epsilon,p}+\ddb \phi$ is K\"ahler (for $\omega_{\epsilon,p}$ the restriction of $\omega_{\epsilon}$ to $\X_p$), let 
\begin{align}
\label{eq:fibrewiseeqn}
h_{u, \epsilon, p, \phi} = h_{u,\epsilon,p} + \frac{1}{2} \langle u , \nabla \phi \rangle,
\end{align}
where $\nabla$ is the gradient with respect to the metric induced by $\omega_{\epsilon,p}+\ddb \phi$. We then define 
$$
\bar \mfk^T_{p,\epsilon, \phi} = \{ h_{u, \epsilon, p, \phi} : u \in \mfk^T \}.
$$
Given a function $\phi_{\epsilon}$ on $\X$ such that $\omega_{\epsilon} + \ddb \phi_{\epsilon}$ is relatively K\"ahler we may then define
$$
\bar \mfk^T_{V,\epsilon} = \{ h \in C^{\infty}(\X) : h_p \in \bar \mfk^T_{\epsilon, p, \phi_{\epsilon,p}} \textnormal{ for all } p \in X \},
$$ where $\phi_{\epsilon,p}$ and $h_p$ the respective restrictions to $\X_p$. Solving $S_V(\omega_{\epsilon} + \ddb \phi_{\epsilon}) \in \bar \mfk^T_{V, \epsilon}$ is then equivalent to solving 
$$
S(\omega_{\epsilon,p} + \ddb \phi_{\epsilon,p}) = h_{u_{\epsilon,p}, \epsilon, p, \phi_{\epsilon,p}}
$$ 
on $\X_p$ for all $p \in X$ in a smoothly varying way: we then define $\phi_{\epsilon}$ as the function which restricts to $\phi_{\epsilon,p}$ for every fibre $\X_p$. Note that the vector field $u_{\epsilon,p}$ is allowed to depend on both $\epsilon$ and $p\in \X^T$.

Note that in the case that $u$ fixes $p$ so that $u$ lifts to $\Bl_p X$, Equation \eqref{eq:fibrewiseeqn} gives an explicit formula for the holomorphy potential of $u$ on $\Bl_pX$ with respect to $\omega_{\epsilon,p}+\ddb \phi$. In particular, for all elements of $\mft$, the two lifts agree for all points $p \in X^T$. However, in the case when the vector field does not lift to the fibre over $p$, the above formula for the potential does not agree with the genuine holomorphy potential with respect to the relatively K\"ahler metric $\omega_{\epsilon} + \ddb \phi_{\epsilon}$ we obtain after gluing the fibrewise solutions together. As $\mfk^T$ may be larger than $\mft$ (and is larger in situations of interest), there may be elements of $\mfk^T$ for which this is the case. We will see in subsequent sections that this approximates the restriction of a genuine holomorphy potential on $\X$ to the fibre to a sufficiently high accuracy for our purposes. 
\begin{theorem}
\label{thm:fibrewisesoln}
For all $0 < \epsilon \ll 1$ and for all $p \in X^T$, there exists $\phi_{\epsilon,p} \in C^{\infty}(\Bl_p X)^T$ and $h_{\epsilon, p} \in \bar \mfk_{p,T,\epsilon, \phi_{\epsilon,p}}$ corresponding to a vector field $u_{\epsilon,p} \in \mfk^T$ such that
$$
S(\omega_{\epsilon,p}+\ddb \phi_{\epsilon,p}) - \frac{1}{2} \langle u_{\epsilon,p}, \nabla(\phi_{\epsilon,p}) \rangle - h_{\epsilon,p} = 0.
$$
For $n>2$, the potential $ \frac{1}{2} \langle u_{\epsilon,p}, \nabla(\phi_{\epsilon,p}) \rangle + h_{\epsilon,p}$ admits an expansion
$$
 \frac{1}{2} \langle u_{\epsilon,p}, \nabla(\phi_{\epsilon,p}) \rangle+ h_{\epsilon,p}= S(\omega) + \epsilon^{2n-2} \left(c \mu(p) + \frac{c}{\Vol(X)} \right) + O(\epsilon^{2n-1}) + \gamma_2 O(\epsilon^2),
$$
where the $O(\epsilon^{2n-1})$-term is over the whole of $\Bl_pX$ and the term $\gamma_2 O(\epsilon^2)$ is an $O(\epsilon^2)$-function supported on $B_{2r_{\epsilon}}$. In the case $n=2$, the same holds except that the $O(\epsilon^{2n-1})$-term is $O(\epsilon^{2n-2+\theta})$ for some $\theta >0$.

Finally, the function $\phi_{\epsilon}$ on $\X^T$ whose restriction to the fibre $\sigma^{-1}(p) \cong \Bl_p X$ is $\phi_{\epsilon, p}$ is smooth.
\end{theorem}
\begin{proof}
The fact that we can solve the equation follows from the contraction mapping theorem as in \cite{gabor-blowups1}. This uses that $Q_{\epsilon,p}$ is a right inverse with bound independently of $\epsilon$ (in the case $n>2$), which follows from Proposition \ref{prop:linearisationinverse} since the approximate solution is a small perturbation of our initial approximate solution, and that by Lemma \ref{lem:approxsoln}, $(0, S(\omega)+\epsilon^{2n-2} h_p) \in \mfk$ is approximately solving the equation. Similarly, for the case $n=2$,  $Q_{\epsilon,p}$ is a right inverse with bound which $O( \epsilon^{\delta})$. Using this, one can show that the relevant operator is a contraction on the set
\begin{align}
\label{eq:contractionset}
\left\{ (f,h) : \|f\|_{C^{4,\alpha}_{\delta}}+|h| \leq c r_{\epsilon}^{4-\delta} \right\}
\end{align}
for a suitably chosen constant $c$ in the case when $n>2$ and on
\begin{align*}
\left\{ (f,h) : \|f\|_{C^{4,\alpha}_{\delta}}+|h| \leq c r_{\epsilon}^{4-\delta} \epsilon^{-\delta} \right\}
\end{align*} 
in the case when $n=2$. The smoothness of the solutions follows from the fact that this right inverse and approximate solutions depend smoothly on $p$. 

The remaining point is to prove that the potentials $\frac{1}{2} \langle u_{\epsilon,p}, \nabla(\phi_{\epsilon,p}) \rangle + h_{\epsilon,p}$ have the required expansion. This follows for the approximate solution $\omega_{\epsilon,p}$ because the $O(1)$-part in $\epsilon$ of the expansion holds  for the initial potential given by Equation \eqref{eq:pots}, and the change we made to the approximate solution $\omega_{\epsilon,p}$ introduced the $\epsilon^{2n-2}$-term. As the solution to the non-linear equation above is found on the set given by Equation \eqref{eq:contractionset}, it follows that this expansion of the holomorphy potentials is preserved upon perturbing from the approximate solution to the actual solution.
\end{proof}

\section{The existence of extremal metrics on blowups}\label{sec4}

\subsection{Moment map geometry}\label{moment-map-geometry} As before, consider the holomorphic submersion $\pi: \X \to X$, with $K\subset \Aut_0(X,\alpha)$ a maximal compact subgroup acting on $\X$ and $X$ in an equivariant manner. We fix a point $p \in X$ and denote by $T$ a maximal subtorus of $K$ fixing $p$. We let $\mft$ and $\mfk$ be the Lie algebras of $T$ and $K$ respectively. As throughout, we consider the fixed point locus $X^T$ of $T$ and the induced submersion $\X^T \to X^T$, where $\X^T$ is the preimage of $X^T$ in $\X$, and let $K^T \subset K$ be the commutator of $T$ in $K$. Our assumption is that there exists an extremal metric in the class $\alpha$, giving an associated extremal vector field. The cscK case is a special case of the extremal case, so we address only the cscK case. We further assume that the extremal vector field $\nabla^{1,0} S(\omega)$ on $X$ actually vanishes at $p$, hence lies in $\mft$, as otherwise it follows that $(\Bl_pX,\alpha_{\epsilon})$ cannot admit an extremal metric for $0<\epsilon\ll 1$ \cite[Proposition 40]{gabor-blowups2}.

We now relabel $\omega_{\epsilon} \in \scA_{\epsilon}$ as the relatively K\"ahler metrics constructed by Theorem \ref{thm:fibrewisesoln}, which are then $K_T$-invariant on $\X^T$.  This relatively K\"ahler metric induces a Hermitian metric on the relative anticanonical class $-K_{\X^T/X^T}$, with curvature which we denote $\rho_{\epsilon}$.

 Let 
 $$\mu_{\epsilon} : \X\to (\mfk^T)^*$$ be a moment map for the $K^T$-action on $\X^T$. While $\omega_{\epsilon}$ may only be relatively K\"ahler, what we mean here is that the usual moment map condition holds, or equivalently that $\omega_{\epsilon}+\mu_{\epsilon}$ is equivariantly closed with respect to the  $K^T$-action. We normalise the moment maps $\mu_{\epsilon}$ such that for each element $u\in \mfk^T$  the function induced by pairing with the moment map has integral zero over each fibre, using the fibrewise volume form induced by $\omega_{\epsilon}$.

We next involve the extremal vector field on $(\Bl_pX,\alpha_{\epsilon})$, which is characterised---without knowing whether this manifold admits an extremal metric---as follows. Note that, since $T$ is a maximal torus in $\Aut_0(X,\alpha)_p$, it induces a maximal torus in $\Aut_0(\Bl_pX,\alpha_{\epsilon})$. For commuting elements $u,u'\in \mfk^T$, we define the ($\epsilon$-dependent) \emph{Futaki invariant} on $(\Bl_pX,\alpha_{\epsilon})$ to be $$\Fut_{\epsilon}(u) = \int_{\Bl_pX} \langle \mu_{\epsilon}, u\rangle \left(\hat S - S(\omega_{\epsilon,p})\right)\omega_{\epsilon,p}^n,$$ where $\omega_{\epsilon,p}$ is the restriction of $\omega_{\epsilon}$ to the fibre $\Bl_pX$ of $\pi$ over $p\in X$, and further define the  (Futaki--Mabuchi) \emph{inner product} by $$\langle u, u'\rangle_{\epsilon} = \int_{\Bl_pX}\langle \mu_{\epsilon}, u\rangle\langle \mu_{\epsilon}, u'\rangle\omega_{\epsilon, p}^n,$$ where we use that by normalisation of the moment map the functions $\langle\mu_{\epsilon}, u\rangle$ and $\langle \mu_{\epsilon}, u'\rangle$ have integral zero over $\Bl_pX$. Both the Futaki invariant and the inner product are independent of choice of $\omega_{\epsilon,p}$ \cite{futakimabuchi95}, and are further independent of $p\in X^T$; the latter independence can be seen either by equivariant differential geometry or a cohomological argument (by involving compactifications of test configurations \cite{legendre}, and using invariance of degrees of differential forms). The extremal vector field of $(\Bl_pX,\alpha_{\epsilon})$ is then defined to be the unique vector field $\xi_{\epsilon} \in \mft$ such that $$\Fut_{\epsilon}(u) = \langle u, \xi_{\epsilon}\rangle_{\epsilon}$$ for all $u\in \Lie T$. It follows that this vector field is also, for each $\epsilon$, independent of $p \in X^T$. As the vertical scalar curvature of $\omega_{\epsilon}$ has an expansion in $\epsilon$ where the first non-constant term comes at order $\epsilon^{2n-2}$ (or as we shall see in Section \ref{GIT-section} through algebraic geometry), we may expand the extremal vector field as   
$$
\xi_{\epsilon} = \xi + \epsilon^{2n-2} \xi' + O(\epsilon^{2n-1}).
$$

Define a $(1,1)$-form on $X^T$ by
$$
\Omega_{\epsilon} = -\int_{\X^T / X^T} \rho_{\epsilon} \wedge \omega_{\epsilon}^n + \frac{1}{n+1}\int_{\X^T/X^T}\left(\hat S_{\epsilon} +  \langle \mu_{\epsilon}, \xi_{\epsilon} \rangle\right) \omega_{\epsilon}^{n+1},
$$
where $\xi_{\epsilon} \in \mft$ is the extremal vector field of the extremal metric on $X$. One checks that $\Omega_{\epsilon}$ is actually closed, either by general theory of equivariant differential geometry (using that the $T$-action on $X^T$ is trivial), or by direct calculation. This is the \emph{Weil--Petersson form} associated to $\X^T\to X^T$.

\begin{theorem}
\label{thm:momentmap}
A moment map for the $K_T$ action on $(X^T,\Omega_{\epsilon})$ is given by the map $\sigma_{\epsilon} : X^T \to \mfk_T^*$ defined by
$$
\langle \sigma_{\epsilon},w\rangle(p) = \int_{\Bl_pX}\langle \mu_{\epsilon} , w \rangle \left( S(\omega_{\epsilon, p}) - \hat S_{\epsilon}- \langle \mu_{\epsilon}, \xi_{\epsilon}\rangle \right) \omega_{\epsilon,p}^n.
$$

\end{theorem}

In the case that the $T$-action is trivial, which is the case related to cscK metrics, this was proven by the first author and Hallam \cite{DH} (as a variant of the classic work of Fujiki and Donaldson \cite{fujiki, donaldson-moment}). The proof in the general case was explained to us by Hallam (along with the definition and closedness of $\Omega_{\epsilon})$---whom we thank---and follows exactly the lines of the proof of \cite{DH}, to where we refer for further details. 

While $\Omega_{\epsilon}$ is a closed $(1,1)$-form, it does not follow from general theory that it is actually K\"ahler. To circumvent this, we next use the expansion of the scalar curvature to understand the moment map $\sigma_{\epsilon}$ in more detail. Recall that 
$$
\xi_{\epsilon} = \xi + \epsilon^{2n-2} \xi' + O(\epsilon^{2n-1}).
$$ 
Write $\langle \cdot,\cdot \rangle_0$ for the Futaki--Mabuchi inner product on vector fields on $X$.

\begin{lemma}\label{lem:momentmapexp} Assume $n>2$. For $0<\epsilon \ll 1$, the moment map $\sigma_{\epsilon}$ satisfies $$\langle\sigma_{\epsilon},u\rangle(p) = \epsilon^{2n-2} \left( \langle\mu,u\rangle(p) - \langle \xi', u\rangle_0 \right) + O(\epsilon^{2n-1}),$$ where $\mu$ is the moment map for the $K_T$-action on $(X,\omega)$ and $ \langle\mu,w\rangle$ denotes the pairing of $\mfk^T$ and its dual.  In the case $n=2$, the same holds except that the $O(\epsilon^{2n-1})$-term is $O(\epsilon^{2n-2+\theta})$ for some $\theta>0$.
\end{lemma}

\begin{proof}
The volume form on the fibres can be written as 
$$
\omega_{\epsilon}^n = \omega^n + \beta_{\epsilon},
$$
for an $(n,n)$-form $\beta_{\epsilon}$ that depends on $\epsilon$. Now, the class of $\omega_{\epsilon}$ on the fibres is $\pi^* [\omega] - \epsilon^2 [E]$, where $[E]$ is the class of the exceptional divisor. This has volume $[\omega]^n + c_n \epsilon^{2n-2}$ for a dimensional constant $c_n$. Thus $\beta_{\epsilon}$ integrates to $c_n \epsilon^{2n-2}$ over the fibres. 

Now, as $\omega$ is extremal on $X$ with extremal vector field $\xi$, 
\begin{align*}
S(\omega_{\epsilon, p}) - \langle \mu_{\epsilon}, \xi_{\epsilon}\rangle - \hat S_{\epsilon} =& \left(S(\omega) - \langle \mu, \xi \rangle - \hat S\right) + O(\epsilon^{2n-2})
\end{align*} 
is $O(\epsilon^{2n-2})$. Thus, 
$$
\int_{\Bl_pX}\langle \mu_{\epsilon} , u \rangle \left( S(\omega_{\epsilon, p}) - \hat S_{\epsilon} - \langle \mu_{\epsilon}, \xi_{\epsilon}\rangle \right) \beta_{\epsilon}
$$
is $O(\epsilon^{4n-4})$, which is better than the $O(\epsilon^{2n-1})$ contribution we require. Similarly, 
$$
\langle \mu_\epsilon, u \rangle = h + \gamma_2 O(1),
$$
where $h$ is the holomorphy potential on $X$ pulled back to $\Bl_p X$, and so 
$$
\int_{\Bl_pX}( \langle \mu_{\epsilon} , u \rangle - h) \left( S(\omega_{\epsilon, p}) - \langle \mu_{\epsilon}, \xi_{\epsilon}\rangle \right) \omega_{\epsilon, p}^n = O(\epsilon^{2n-1})
$$
as $\gamma_2$ has support the ball of radius $2r_{\epsilon}$. The upshot is that to establish the result we must show that the required expansion holds for 
$$
\int_{\Bl_pX} h (S(\omega_{\epsilon, p}) - \langle \mu_{\epsilon}, \xi_{\epsilon}\rangle  - \hat S_{\epsilon})  \omega^n.
$$

The subleading order term in the expansion of the scalar curvature of $\omega_{\epsilon,p}$ is
$$
\epsilon^{2n-2} \left(c \mu(p) + \frac{c}{\Vol(X)} \right).
$$
This gives the term $c \langle \mu, u \rangle (p) \epsilon^{2n-2}   + O(\epsilon^{2n-1})$ in the required expansion when $n>2$, the constant term canceling with the subleading order term in the expansion of $\hat S_{\epsilon}$. The error term is $O(\epsilon^{2n-2+\theta})$ in the case $n=2$, leading to the slightly different expansion in that case. If we let $\xi_{\epsilon} = \xi + \epsilon^{2n-2} \xi' + O(\epsilon^{2n-1})$ denote the expansion of $\xi_{\epsilon}$ (with $2n-1$ replaced by $2n-2+\theta$ in the case $n=2$), we are therefore left with considering 
\begin{align}
\label{eq:LTP}
\int_{\Bl_pX} h (\langle \mu, \epsilon^{2n-2} \xi' \rangle + \langle \mu_{\epsilon} - \mu, \xi_{\epsilon}\rangle)  \omega^n.
\end{align}

For the first part of Equation \eqref{eq:LTP}, we note that
$$
\int_{\Bl_pX} h \langle \mu, \epsilon^{2n-2} \xi'\rangle \omega^n = \epsilon^{2n-2} \int_{X} h \langle \mu,  \xi' \rangle \omega^n,
$$
as everything involved is pulled back from $X$. But as $h=\langle \mu, u \rangle$, this is nothing but 
$$\epsilon^{2n-2} \int_{X} \langle \mu, u\rangle \langle \mu,  \xi' \rangle \omega^n = \epsilon^{2n-2}\langle u, \xi' \rangle,
$$ giving the corresponding term in the expansion. 

For the second part of Equation \eqref{eq:LTP}, we only need to consider 
$$
\int_{\Bl_pX} h  \langle \mu_{\epsilon} - \mu, \xi\rangle  \omega^n
$$
as $\xi_{\epsilon} - \xi$ is $O(\epsilon^{2n-2})$. Moreover, $\langle \mu_{\epsilon} - \mu, \xi\rangle$ is the change in the Hamiltonian for $\xi$ compared to the pull back of the Hamiltonian on $X$. From the description in Equation \eqref{eq:pots} of the change in the Hamiltonians this is $O(1)$ and $\langle \mu_{\epsilon} - \mu, \xi\rangle$ is supported on the region of radius $2r_{\epsilon}$ about the exceptional divisor. The volume of this region with respect to the pullback of $\omega$ is $O(r_{\epsilon}^{2n})$, i.e. $\epsilon^{2n-\frac{2n}{2n+1}}$, which in particular is $O(\epsilon^{2n-1})$, including in dimension $2$, which in particular makes it $O(\epsilon^{2n-2+\theta})$ in this case. Thus this term does not contribute to the order $\epsilon^{2n-2}$-term. 

Putting together all of the above gives the required expansion.
\end{proof}

We argue similarly for the forms $\Omega_{\epsilon}$.

\begin{lemma}
\label{lem:formexp} We may write $$\Omega_{\epsilon}  =  \Omega_{\epsilon, 2n-2} + \Omega_{\epsilon, 2n-1},$$ where $\Omega_{2n-2} = \epsilon^{2n-2} \tilde \Omega_{2n-2}$ for a fixed $(1,1)$-form and $\Omega_{\epsilon, 2n-1}$ is $O(\epsilon^{2n-1})$ when $n>2$ and $O(\epsilon^{2n-2+\theta})$ when $n=2$.

\end{lemma}

\begin{proof}

We argue similarly to Lemma \ref{lem:momentmapexp}. Explicitly, $\Omega_{\epsilon}$ is defined as 
\begin{align*}
\Omega_{\epsilon} =& -\int_{\X^T / X^T} \rho_{\epsilon} \wedge \omega_{\epsilon}^n + \frac{1}{n+1}\int_{\X^T/X^T}\left(\hat S_{\epsilon} +  \langle \mu_{\epsilon}, \xi_{\epsilon} \rangle\right) \omega_{\epsilon}^{n+1}.\end{align*}
Let $\alpha_{V}$ and $\alpha_{H}$ denote respectively the vertical and horizontal parts of a $(1,1)$-form $\alpha$ on $\X^T$. Then 
$$
\omega_{\epsilon}^{n+1} = (n+1) (\omega_{\epsilon})_{H} \wedge (\omega_{\epsilon})_{V}^n
$$
and 
$$
\rho_{\epsilon} \wedge \omega_{\epsilon}^{n} =  (\rho_{\epsilon})_{H} \wedge (\omega_{\epsilon})_{V}^n + S_{V}(\omega_{\epsilon}) (\omega_{\epsilon})_{H} \wedge (\omega_{\epsilon})_{V}^n.
$$
As before,
$$
(\omega_{\epsilon})_{V}^n = \omega^n + \beta_{\epsilon},
$$
where $\beta_{\epsilon}$ integrates to a constant multiple of $\epsilon^{2n-2}$. Moreover, by construction,
$$
S_{\V} (\omega_{\epsilon}) = \left(\hat S_{\epsilon} +  \langle \mu_{\epsilon}, \xi_{\epsilon} \rangle\right) + O(\epsilon^{2n-2})
$$
and since $\omega$ pulled back to $X \times X$ and then to $\X$ is purely vertical
$$
(\omega_{\epsilon})_{\H} = \epsilon^2 \ddb\left(\gamma_2 \cdot (\gamma(\epsilon^{-1}d)\log (\epsilon^{-2}d^2) + f(\epsilon^{-2}d^2) + \epsilon^{2n-2} \gamma_1 \Gamma )\right)_H
$$
up to a $\ddb$ of a term which is $O(r_{\epsilon}^{4-\delta})$ in the $C^{4,\alpha}_{\delta}$-norm, which in particular means it is a sum of a term supported in the region $d \leq 2 r_{\epsilon}$ and which decays with $\epsilon$, and a term which is $O(\epsilon^{\lambda})$  for a $\lambda > 2n-2$ on the whole of the blowup, and so will give a term which is of strictly higher order in $\epsilon$ upon integration over the fibres. Here we use that $\delta$ is chosen very close to $4-2n$, so that $r_{\epsilon}^{4-\delta} \leq \epsilon^{\lambda}$ for a suitable chosen $\lambda > 2n-2$, see the proof of \cite[Proposition 15]{gabor-blowups1}. Thus
\begin{align*}
 &-\int_{\X^T / X^T} S_{V}(\omega_{\epsilon}) (\omega_{\epsilon})_{H} \wedge (\omega_{\epsilon})_{V}^n + \frac{1}{n+1}\int_{\X^T/X^T}\left(\hat S_{\epsilon} +  \langle \mu_{\epsilon}, \xi_{\epsilon} \rangle\right) \omega_{\epsilon}^{n+1} \\
 =&    \int_{\X^T/X^T}\left(\hat S_{\epsilon} +  \langle \mu_{\epsilon}, \xi_{\epsilon} \rangle - S_{V} (\omega_{\epsilon}) \right)(\omega_{\epsilon})_H \wedge (\omega_{\epsilon})_V^{n} \\
 =& O(\epsilon^{4n-4}),
\end{align*}
since the region where $d \leq 2 r_{\epsilon}$ has volume which is $O(\epsilon^{2n-2})$ with respect to $\omega^n$ and has volume which is even higher order in $\epsilon$ with respect to $\beta_{\epsilon}$.

This leaves the term
$$
\int_{\X^T/X^T} (\rho_{\epsilon})_{H} \wedge (\omega_{\epsilon})_{V}^n.
$$
While the class of $\rho_{\epsilon}$ is fixed independently of $\epsilon >0$, this term is $O(\epsilon^{2n-2})$ since $(\rho_{\epsilon})_{H}$ is supported in the region $d \leq 2 r_{\epsilon}$, just as was the case for $(\omega_{\epsilon})_H$, leading to an $O(\epsilon^{2n-2})$-term upon integration. The term $\epsilon^{2n-2} \gamma_1 \Gamma$ may also affect this term, as it is $O(\epsilon^{2n-2})$. All other terms are strictly higher order in $\epsilon$, again coming from the fact that we perturb the second approximate solution by a term in the $C^{4,\alpha}_{\delta}$ weighted space, and from the fact that the difference of the vertical Ricci forms associated to two relatively K\"ahler metrics on $\X^T$ are given by a power series expansion in the potential function. \end{proof}

 Denote 
$$
\Omega'_{\epsilon} = \epsilon^{2n-2} \omega +  \Omega_{\epsilon, 2n-1} 
$$ 
which then is clearly K\"ahler for $0<\epsilon \ll 1$. 
\begin{corollary} The map $\sigma_{\epsilon}$ is also a moment map for the $K_T$-action on $(X^T,\Omega_{\epsilon}')$.\end{corollary}
\begin{proof}
We only need to verify that the Hamiltonian condition holds with respect to $\Omega_{\epsilon}'$ as well, since the equivariance property of the moment map does not involve the two-form. We know that 
$$
d \langle\sigma_{\epsilon},u\rangle = - \iota_{u} \Omega_{\epsilon}
$$
for all sufficiently small $\epsilon$, since $\sigma_{\epsilon}$ is a moment map with respect to $\Omega_{\epsilon}$ by Theorem \ref{thm:momentmap}. We can expand the moment map as
$$
\epsilon^{2n-2} \left(\langle\mu,u\rangle(p) - \langle \xi',u\rangle\right) + \langle \sigma_{\epsilon, 2n-1}, u\rangle,
$$
where $\sigma_{\epsilon, 2n-1}(t)$ is $O(\epsilon^{2n-1})$ if $n>2$ and $O(\epsilon^{2n-2+\theta})$ if $n=2$. Comparing the expansions of $\sigma_{\epsilon}$ and $\Omega_{\epsilon}$, and using that this holds for all sufficiently small $\epsilon$, we see that 
\begin{align*}
d  \left(\langle\mu,u\rangle - \langle \xi', u\rangle\right)  = - \iota_{u} \tilde \Omega_{2n-2}
\end{align*}
and 
\begin{align*}
d  \langle \sigma_{\epsilon, 2n-1},u\rangle  = - \iota_{u} \Omega_{\epsilon, 2n-1}
\end{align*}

Let $(\xi')^*$ denote the dual element of $(\mfk^T)^*$ defined through the inner product $\langle \cdot, \cdot\rangle_0$, which is defined independently of $p$. The map $\mu$ is a moment map for the $K_T$-action on $X^T$, and we have added a constant element in $\mfk^*$ to $\mu$, preserving the Hamiltonian property for being a moment map. Since $\xi' \in \mft$ is a central element of $\mfk_T$, the $K_T$-invariance is preserved as well. Thus $\mu + (\xi')^*$ is a moment map with respect to $\omega$ and hence we can replace $\Omega_{\epsilon, 2n-2} = \epsilon^{2n-2} \tilde \Omega_{2n-2}$ with $\epsilon^{2n-2}$ and retain the moment map property for $\sigma_{\epsilon}$, giving that $\sigma_{\epsilon}$ is a moment map with respect to $\Omega_{\epsilon}'.$
\end{proof}

Define the function space $\bar \mfk^T_{\pi,\epsilon}$ on $\X^T$ by taking the space of $\langle \mu_{\epsilon}, u\rangle$ over all $u \in \mfk^T$.  A zero of the moment map $\sigma_{\epsilon}$ satisfies the condition that the function $S(\omega_{\epsilon, p}) - \langle \mu_{\epsilon}, \xi_{\epsilon}\rangle$ is $L^2$-orthogonal to the restriction of $\bar\mfk^T_{\pi,\epsilon}$ to $\Bl_p X$. This is a different function space to the space $\bar \mfk^T_{V,\epsilon}$ involved in the gluing argument, requiring us to establish the following. The argument follows the argument for a similar statement in \cite[Lemma 4.15]{OS}.

\begin{lemma}\label{OS-lemma} For all $0<\epsilon \ll 1$, a zero of the moment map $\sigma_{\epsilon}$ is an extremal metric.

\end{lemma}

\begin{proof}
As explained above, $p$ is a zero of the moment map $\sigma_{\epsilon}$ if and only if $S(\omega_{\epsilon, p}) - \langle \mu_{\epsilon}, \xi_{\epsilon}\rangle$ is $L^2$-orthogonal to the restriction of $\bar\mfk^T_{\pi,\epsilon}$ to $\Bl_p X$. In other words, a zero of the moment map then corresponds to a point $p$ such that the $L^2$-orthogonal projection of $S(\omega_{\epsilon, p}) - \langle \mu_{\epsilon}, \xi_{\epsilon}\rangle$ to $\bar \mfk^T_{\pi,\epsilon}\big|_{\Bl_pX}$ is zero. Let 
$$\Pi_{\epsilon,p} :\bar \mfk^T_{V,p,\epsilon} \to \bar\mfk^T_{\pi,\epsilon}\big|_{\Bl_pX}$$ denote the $L^2$-orthogonal projection with respect to $\omega_{\epsilon,p}$, with $\mfk^T_{V,p,\epsilon}$ the restriction of this function space to $\Bl_pX$. The claim will then follow if we show this is an isomorphism. 

Let $u_1, \ldots, u_k$ denote a basis of $\mfk^T$ such that the corresponding hamiltonians $h_j \in \bar \mfk_T \subset C^{\infty}(X)$ are orthonormal with respect to $\omega$ (to be more precise, these $h_j$ are defined by $h_j=\langle \mu, u_j\rangle$). In turn, let $h_{\epsilon,j}$ denote the corresponding holomorphy potentials with respect to $\omega_{\epsilon}$ on $\X$ (so defined via $\mu_{\epsilon})$). A general element $u$ of $\bar \mfk^T_{V,p,\epsilon}$ can then be written as 
\begin{align*}
u= \sum_{j=1}^k a_j \left(h_{\epsilon,j} + u_j(\phi_{\epsilon,p})\right)
\end{align*}
for constants $a_j$, where we implicitly restrict $h_{\epsilon,j}$ to $\Bl_p X$. Similarly, a basis of $\bar \mfk^T_{\pi,\epsilon}$ is given by the $(h_{\epsilon,j} + u_j(\phi_{\epsilon}))\big|_{\Bl_pX}$, and so the orthogonal projection of $u$ is given by
\begin{align*}
\sum_{i,j=1}^k a_i \frac{\left(\int_{\Bl_p X} \left(h_{\epsilon,i} + u_i(\phi_{\epsilon,p})\right)\omega_{\epsilon,p}^n \right) \left( \int_{\Bl_p X} (h_{\epsilon,j} + u_j(\phi_{\epsilon})) \omega_{\epsilon,p}^n \right)}{\left( \int_{\Bl_p X}(h_{\epsilon,j} + u_j(\phi_{\epsilon}))^2 \omega_{\epsilon,p}^n \right)^{\frac{1}{2}}}.
\end{align*}

The $O(1)$ contribution of all of the terms involved of the above equals that at $t=0$, as in the expansions considered in Lemma \ref{lem:momentmapexp}.
Thus 
$$
\Pi_{\epsilon,p}(u) = u + O(\epsilon^{2n-1}),
$$
which implies that $\Pi_{\epsilon,p}$ is an isomorphism for all sufficiently small $\epsilon$ and hence the result.
\end{proof}

We have thus reduced the problem to finding a zero of the moment map  $\sigma_{\epsilon}$  on $\left(X^T, \Omega_{\epsilon}'\right)$, where we  fix an $\epsilon$ such that Lemma \ref{OS-lemma} applies. We may now appeal to the Kempf--Ness theorem to relate this to algebro-geometric stability.  Letting $G^T$ denote the complexification of $K^T$, the Kempf--Ness theorem states the following (see for example \cite{teleman} or the book \cite{moment-weight}).

\begin{theorem}\label{KN} Precisely one of the following holds:
\begin{enumerate}
\item  there is a zero of the moment map $\sigma_{\epsilon}$ in the orbit $G_T.p$;
\item there is an element $u \in \mfk^T$ such that, denoting $$q = \lim_{s\to\infty} \exp(isu).p$$ the specialisation of $p$, the numerical inequality $$\langle \sigma_{\epsilon}, u\rangle (q)<0$$ is satisfied. Furthermore, $u$ can be taken to be rational, in the sense that its flow produces an $S^1$-action on $X^T$.
\end{enumerate}
\end{theorem}

Applying this to the precise definition of the moment map $\sigma_{\epsilon}$, Lemma \ref{OS-lemma} and Theorem \ref{KN} produces the following dichotomy.

\begin{corollary}\label{first-characterisation}
Either $(\Bl_pX, \alpha_{\epsilon})$ admits an extremal metric for all $0<\epsilon \ll 1$, or there is a $u \in \mfk^T$ which specialises $p$ to $q$ such that $$\int_{\Bl_qX}\langle \mu_{\epsilon} , u \rangle \left( S(\omega_{\epsilon, q}) - \hat S_{\epsilon} -   \langle \mu_{\epsilon}, \xi_{\epsilon}\rangle \right) \omega_{\epsilon,q}^n<0.$$
\end{corollary}

\subsection{Relative K-stability as the obstruction}\label{sec:rel-kstable} By this stage, we have realised the obstruction to the existence of an extremal metric as an element  $u \in \mfk^T$ destabilising $p$ in the sense of geometric invariant theory. We next relate this obstruction to relative K-stability, so that we characterise the existence of extremal metrics on $(\Bl_pX, \alpha_{\epsilon})$ through relative K-stability. Our discussion is rather simple, as all objects involved in our discussion are smooth; in general one must allow singular degenerations.

The definitions given here apply to an arbitrary compact complex manifold $Y$ with a K\"ahler class $\beta$, though we will later apply the definition to $(\Bl_pX, \alpha_{\epsilon})$. Fix a maximal torus $T^{\C}\subset  \Aut_0(Y,\beta)$, commuting vector fields $u,u' \in \Lie T$, a K\"ahler metric $\omega\in\beta$ which is $T$-invariant and a moment map $\mu: Y \to \mft^*$ satisfying the condition that for all $u \in \mft$ the integral $\int_Y \langle \mu, u\rangle\omega^n$ vanishes. We define relative  K-stability.

\begin{definition} A \emph{test configuration} for $(Y,\beta)$ a collection $\pi: (\Y,\scB) \to \C$  with
\begin{enumerate}
\item  $\Y$ a complex manifold  endowed with a  $\C^*$-action induced by a vector field $u$ on $\Y$;
\item $\scB$ a $(1,1)$-class which is relatively K\"ahler  and $\C^*$-invariant;
\item $\pi$ a proper surjective $\C^*$-equivariant holomorphic submersion;
\item all fibres $(\Y_s, \scA_s)$ over $s\neq 0$ are isomorphic to $(Y,\beta).$
\end{enumerate}
For a complex torus $T^{\C} \subset \Aut_0(Y,\beta)$, we say that $(\Y,\scB)$ is $T^{\C}$\emph{-invariant} if there is a $T^{\C}$ action on $(\Y,\scB)$ making $\pi$ a $T^{\C}$-invariant morphism.
\end{definition}

\begin{definition} The  \emph{relative Donaldson--Futaki invariant} of $(\Y,\scB)$ is given by $$\DF_{T^{\C}}(\Y,\scB)=\Fut(u) - \langle u, \xi\rangle,$$ where these quantities are computed on $\Y_0$, and similarly the \emph{Donaldson--Futaki invariant} to be $$\DF(\Y,\scB) = \Fut(u).$$ We then define $(Y,\beta)$ to be 
\begin{enumerate}
\item   \emph{relatively K-stable} if for all $T^{\C}$-equivariant test configurations $(\Y,\scB)$  we have $\DF_{T^{\C}}(\Y,\scB)\geq 0$, with equality if and only if $(\Y_0,\scB_0)\cong (Y,\beta)$;
\item \emph{K-stable} if for all test configurations $(\Y,\scB)$  we have $\DF(\Y,\scB)\geq 0$, with equality if and only if $(\Y_0,\scB_0)\cong (Y,\beta)$.
\end{enumerate}
\end{definition}

\begin{remark}\label{def-using-inner-product}
One may calculate that  $$\DF_{T^{\C}}(\Y,\scB) = \DF(\Y,\scB) - \frac{\langle u,\xi \rangle}{\langle \xi, \xi\rangle}\Fut(\xi),$$ using that $\xi$ is the extremal vector field.
\end{remark}

It follows from general results around relative K-stability and extremal metrics that if  $(\Bl_pX, \alpha_{\epsilon})$ admits an extremal metric, then it is relatively K-stable \cite{stoppa, stoppa-szekelyhidi, relative}. But we see by construction that the element $w \in \mfk_T$ produced by Corollary \ref{first-characterisation} produces a test configuration with central fibre $(\Bl_qX,\alpha_{\epsilon})$, and the weight $\langle \sigma_{\epsilon}, w\rangle (q)$ is precisely the relative Donaldson--Futaki invariant of the resulting test configuration. So we obtain the following characterisation of the existence of extremal metrics on the blowup.

\begin{corollary} There is an $\epsilon_0$ such that for all $0<\epsilon<\epsilon_0$ the following are equivalent:
\begin{enumerate}
\item $(\Bl_pX, \alpha_{\epsilon})$ admits an extremal metric;
\item $(\Bl_pX, \alpha_{\epsilon})$ is relatively K-stable.
\end{enumerate}
\end{corollary}

This proves one part of Theorem \ref{intromainthm-stable}. Note that the analogous cscK statement follows as a simple consequence of the fact that an extremal manifold has vanishing Futaki invariant (hence vanishing extremal vector field) if and only if it is cscK.

\subsection{K-stability and GIT}\label{GIT-section} We next turn to the algebro-geometric aspect of our arguments, where we compute the weight $\langle \sigma_{\epsilon}, w\rangle (q)$ produced by Corollary \ref{first-characterisation} explicitly, or equivalently where we compute the relative Donaldson--Futaki invariant of the associated test configuration.  This only involves the vector field $u$ on the central fibre $\Bl_q \X_0$, and so to ease notation we simply replace $\X_0$ with $X$, $\alpha_0$ with $\alpha$ and $p_0$ with $p$, so that the vector field $u$ vanishes at $p$. 

To understand the relative Donaldson--Futaki invariant of the blow up, we take the approach of Remark \ref{def-using-inner-product}, meaning we must understand how both the Donaldson--Futaki invariant and the inner product vary with $\epsilon$. For the Donaldson--Futaki invariant, this has already been fully understood by Sz\'ekelyhidi \cite{gabor-blowups1,gabor-blowups2} using the strategy of Stoppa \cite{stoppa}. As we will need to extend these arguments to also allow understanding of the inner product, and since one needs a new idea in considering the inner product, we go through the arguments.

We first consider the projective case, with $L$ an ample line bundle such that $c_1(L) = \alpha$, and such that $\epsilon$ is rational, as  we will be able to reduce to this case. While we are interested in the blowup $(\Bl_pX,L-\epsilon^2 E)$, we begin by recalling how to express the various invariants algebraically on $(X,L)$; we will use the corresponding formulas for the blowup to obtain the required results. 

The $\C^*$-action on $(X,L)$ induces a $\C^*$-action on the vector spaces $V_k = H^0(X,kL)$ for all $k>0$. Standard results of Donaldson imply that the Donaldson--Futaki invariant and the inner product can be understood from these $\C^*$-actions in the following manner. Suppose the $\C^*$-action on $V_k$ diagonalises as $(t^{\lambda_1},\hdots,t^{\lambda_{N_k}})$, with $N_k = \dim V_k$. Consider the three functions $a(k), b(k)$ and $c(k)$ defined by  $$a(k) = \dim V_k, \qquad b(k) = \sum_{j=1}^{N_k} \lambda_j, \qquad c(k) = \sum_{j=1}^{N_k} \lambda_j^2.$$ Then these are  polynomials with rational coefficients of degree $n, n+1, n+2$ respectively for $k \gg 0$, which we write \begin{align*} a(k) &= a_0 k^n + a_1 k^{n-1} + O(k^{n-2}), \\ b(k) &=  b_0 k^{n+1} + b_1 k^{n} + O(k^{n-1}), \\ c(k) &= c_0k^{n+2} + O(k^{n+1}).\end{align*} 

For a general vector space $V$ of dimension $N$ with a $\C^*$-action, we also use the notation $\wt V = \sum_{j=1}^{N} \lambda_j$ for the \emph{total weight}, and $\sq V = \sum_{j=1}^{N} \lambda_j^2$ for the \emph{squared weight}, so that $a(k) = \dim V_k, b(k) = \wt V_k$ and $c(k) = \sq V_k$.

\begin{proposition}\cite{donaldson,donaldson-lower}\label{donaldson-ag-dg} The Futaki invariant and norm are given by $$F(u) = 4\frac{b_0a_1 - b_1a_0}{a_0}, \qquad \|u\|_2^2 = \frac{c_0a_0 - b_0^2}{a_0^2}.$$\end{proposition}

The polarisation identity $$2\langle u, v \rangle = \|u+v\|^2 - \|u\|^2 - \|v\|^2$$ will ultimately allow us to understand the inner products by understanding the norms, so we only consider the norms for the moment.

Let us now write $a_0(\epsilon), b_0(\epsilon)$ and $c_0(\epsilon)$ for the corresponding numerical invariants calculated with respect to the induced $\C^*$-action on the vector spaces $H^0(\Bl_p X, k(L-\epsilon E))$, where we only consider $k,\epsilon$ such that $k\epsilon$ is an integer. In order to understand their dependence on $\epsilon$, we also consider the $\C^*$-action on the one-dimensional vector space $L_p$ and on the $n$-dimensional vector space $T^*_p X$.

\begin{lemma}\label{lem:dg-interp} Let $h$ be a Hamiltonian on $X$ for the vector field $u$.  Then \begin{align*} \wt L_p &= -h(p), \\ \wt T^*_p X &= -\Delta h(p), \\ \sq  T^*_p X &= \langle \Hess(h), \Hess(h)\rangle(p),\end{align*} where $\Hess(H)$ denotes the Hessian.\end{lemma}

\begin{proof} The fact that $\wt L_p = -h(p)$ is standard, while the additional observation that $\wt T^*_p X = -\Delta h(p)$ is due to Sz\'ekelyhidi \cite[Lemma 28]{gabor-blowups1}, and follows from the fact that the action on $T_pX$ is induced by the Hessian of $h$ at $p$. The same reasoning shows that $\sq  T^*_p X = \langle \Hess(h), \Hess(h)\rangle(p)$. \end{proof}

\begin{proposition}\label{prop:algebraic-expansions}We have expansions \begin{align*} a_0(\epsilon) &= a_0 - \frac{\epsilon^{2n}}{n!}, \\  a_1(\epsilon) &= a_1 - \frac{\epsilon^{2n-2}}{2(n-1)!}, \\ b_0(\epsilon) &= b_0 + \frac{\epsilon^{2n}}{n!}h(p) +\frac{\epsilon^{2n+2}}{(n+1)!}\Delta h(p), \\  b_1(\epsilon) &= b_1+ \frac{\epsilon^{n-1}}{2(n-2)!}h(p) + \frac{(n-2)\epsilon^{2n+2}}{2n!}\Delta h(p),\\ c_0(\epsilon) &=  c_0 - \epsilon^{2n} \frac{h(p)^2}{n!} - \epsilon^{2n+2}  \frac{2h(p)\Delta h(p)}{(n+1)!} -\epsilon^{2n+4} \frac{ \langle \Hess(h), \Hess(h)\rangle(p) + h(p)^2}{(n+2)!}.\end{align*}

\end{proposition} \noindent Note that $h(p)^2 = \sq L_p$, since $L_p$ is one-dimensional.

\begin{proof}

The expansions of $a_0(\epsilon), a_1(\epsilon), b_0(\epsilon)$ and $b_1(\epsilon)$ are due to  Sz\'ekelyhidi \cite[Lemma 28]{gabor-blowups1}, refining work of Stoppa \cite{stoppa}, and the starting point of our approach is the same as in their work. 

Denoting by $\I_{p}$ the ideal sheaf of the point $p$, the isomorphism $$H^0(\Bl_p X, k(L-\epsilon^2 E)) \cong H^0(X, kL\otimes \I_{k\epsilon^2 p})$$ allows us to reduce to understanding the action on $H^0(X, kL\otimes \I_{k\epsilon^2 p})$. Here, as above, we only consider $k,\epsilon^2$ such that $k\epsilon^2$ is an integer. The short exact sequence $$0 \to  kL\otimes \I_{k\epsilon^2 p} \to kL \to \scO_{k\epsilon^2 p} \otimes kL_p\to 0$$ induces for $k \gg 0$ a short exact sequence $$ 0 \to H^0(X, kL\otimes \I_{k\epsilon^2 p}) \to H^0(X,kL) \to \scO_{k\epsilon^2 p} \otimes kL_p \to 0,$$ where we think of the latter as a vector space and where, slightly abusively, $kL_p$ denotes the fibre of $L^{\otimes k}$ over $p$. By equivariance of the short exact sequence, it is enough to understand the action on the vector space $\scO_{k\epsilon^2 p} \otimes kL|_p$. 

Since the vector space $kL_p$ is one-dimensional, the weight on $\scO_{k\epsilon^2 p} \otimes kL_p$ is given by $$\wt(\scO_{k\epsilon^2 p} \otimes kL_p) = \wt(\scO_{k\epsilon^2 p}) + k\wt(L_p)\dim \scO_{k\epsilon^2 p} ,$$ while similarly the square weight is given by $$\sq(\scO_{k\epsilon^2 p} \otimes kL_p) = \sq(\scO_{k\epsilon^2 p}) + k^2\wt L_p^2\dim \scO_{k\epsilon^2 p} + 2k\wt L_p\wt(\scO_{k\epsilon^2 p}).$$
We next turn to the action on $\scO_{k\epsilon^2 p}$. Setting $k\epsilon^2 = l$, similarly to Sz\'ekelyhidi we think of $\scO_{lp}$ as the space of $(l-1)$-jets of functions at $p$, so that $$\scO_{lp} = \C \oplus T^*_pX \oplus \hdots \oplus S^{l-1}T^*_pX.$$ Denoting $V = T^*_pX$, the dimension and total weight satisfy $$\dim S^jV = {n+j -1 \choose j}, \qquad \wt S^j V ={n+j-1 \choose j-1} \wt V,$$ which gives \begin{align*}\dim \scO_{lp} &= \sum_{j=0}^{l-1}\dim S^jV  =  {n+l-1 \choose n} = \frac{1}{n!}\left(l^n + \frac{n(n-1)}{2} l^{n-1}\right) + O(l^{n-2}), \\  \wt \scO_{lp} &= {n+l-1 \choose n+1}\wt V = \frac{\wt V }{(n+1)!}\left(l^{n+1} + \frac{(n+1)(n-2)}{2} l^{n}\right) + O(l^{n-1}).\end{align*}  Summing these formulae reproduces the formulae for $ a_0(\epsilon), a_1(\epsilon), b_0(\epsilon), b_1(\epsilon)$. 

There seems to be no analogous formula for $\sq(\scO_{lp})$ in terms of $\sq V$, so we use a more geometric argument to calculate its leading order term in $l$. Consider the $(n-1)$-dimensional variety $\pr(V)$ with its induced $\C^*$-action. If we denote $$e(j) = \sq(H^0(\pr(V), \scO(j)) = e_0j^{n+1} + O(j),$$ then we know that $e(j)$ is a polynomial for all $j \gg 0$ with leading order term $$e_0 = \int_{\pr(V)} h_u^2 \omega_{FS}^{n-1}$$ where if we diagonalise so that the action is given by $(t^{\gamma_1},\hdots,t^{\gamma_n})$ the function $H_u$ is given by $$h_u = \frac{\sum_{j=1}^n \gamma_j|z_j|^2}{|z|^2}.$$ We can explicitly calculate this integral (see for example \cite[Proposition 3.1.1]{seto}), giving for $i \neq j$ $$\int_{\pr(V)} \frac{|z_i|^2|z_j|^2}{|z|^4} \omega_{FS}^{n-1} = \frac{1}{(n+1)!}, \qquad \int_{\pr(V)} \frac{|z_i|^4}{|z|^4} \omega_{FS}^{n-1} = \frac{2}{(n+1)!},$$ and so $$e_0 = \frac{1}{(n+1)!}(\sq V + (\wt V)^2).$$

As we are interested in the asymptotics of $\sq(\scO_{lp})$, we may assume that $e(j)$ is actually a polynomial for all $j$. Then we see that $$\sum_{j=0}^{l-1} e(j) = e_0\sum_{j=0}^{l-1}j^{n+1} + O(l^{n+1}) = \frac{e_0}{n+2}l^{n+2} + O(l^{n+1}),$$ so $$\sq(\scO_{k\epsilon^2 p}) =k^{n+2}\epsilon^{2n+4} \frac{(\sq V + (\wt V)^2)}{(n+2)!} + O(k^{n+1}).$$ It follows that \begin{align*}&\sq(\scO_{k\epsilon^2 p} \otimes kL_p) = \sq(\scO_{k\epsilon^2 p}) + k^2\wt L_p^2\dim \scO_{k\epsilon^2 p} + 2k\wt L_p\wt(\scO_{k\epsilon^2 p}), \\ &= k^{n+2}\left(\epsilon^{2n+4} \frac{(\sq V + (\wt V)^2)}{(n+2)!} +\frac{\wt L_p^2}{n!}\epsilon^{2n} + \frac{2\wt L_p\wt V}{(n+1)!} \epsilon^{2n+2} \right) + O(k^{n+1}).\end{align*} Finally this means that $$c_0(\epsilon) = c_0 - \epsilon^{2n} \frac{\wt L_p^2}{n!} - \epsilon^{2n+2}  \frac{2\wt L_p\wt V}{(n+1)!} -\epsilon^{2n+4} \frac{(\sq V + (\wt V)^2)}{(n+2)!}, $$ which using Lemma \ref{lem:dg-interp} proves the result.
 \end{proof}
 
 \begin{remark} \label{rmk:laplace}
 
The value $\Delta h(p)$ is simply the weight of the $\C^*$-action on the fibre of the line bundle $K_X$ over $p$, and hence has a natural interpretation in terms of geometric invariant theory. The reason is that if $\mu$ is a moment map with respect to $\omega$, then $\Delta \mu$ is a moment map with respect to the Ricci curvature $\Ric \omega \in c_1(-K_X)$ (in the sense of equivariant differential geometry), see for example \cite[Lemma 28]{gabor-blowups1} and \cite[Proposition 3.5]{JBM}.
 \end{remark}
 
In order to understand the inner product, suppose we have two commuting $\C^*$-actions on  $(X, L)$ fixing $p$ and hence inducing actions  $\lambda$ and $\gamma$ on  $(\Bl_p X, L-\epsilon E)$, and let their Hamiltonians be $h_{\lambda}$ and $h_{\gamma}$ with respect to $\omega$. The key invariant in defining the inner product is defined as follows.  Diagonalise the two one-parameter subgroups as $(t^{\lambda_1},\hdots,t^{\lambda_{N_{k,\epsilon}}})$ and  $(t^{\gamma_1},\hdots,t^{\gamma_{N_{k,\epsilon}}})$ respectively, then define $d_0$ by $$\sum_{j=1}^{N_{k,\epsilon}} \lambda_j \gamma_j = d_0(\epsilon)k^{n+2} + O(k^{n+1}).$$ The \emph{inner product} is then defined to be $$\langle \lambda, \gamma \rangle  = \frac{d_0a_0 - b_{0,\lambda}b_{0,\gamma}}{a_0^2};$$ this agrees with the (Futaki--Mabuchi, $L^2$) inner product of Hamiltonians normalised to integrate to zero. The following is an immediate consequence.

\begin{corollary}\label{cor:ip-expansion}  We have \begin{align*}d_0(\epsilon) = d_0 - \epsilon^{2n} \frac{ h_{\lambda}(p)h_{\gamma}(p)}{n!} -  \epsilon^{2n+2}  &\frac{h_{\lambda}(p)\Delta h_{\gamma}(p)+g_{\gamma}(p)\Delta h_{\lambda}(p)}{(n+1)!} \\  -\epsilon^{2n+4} &\frac{ \langle \Hess(h_{\lambda}), \Hess(h_{\gamma})\rangle(p) + h_{\lambda}(p)h_{\gamma}(p)}{(n+2)!},\end{align*} where $d_0$ denotes the corresponding term computed on $(X, L)$.\end{corollary}

We next show that one can reduce to the projective, rational case.

\begin{proposition} The formulae of Corollaries \ref{futaki-expansion} and \ref{cor:ip-expansion} hold for arbitrary compact K\"ahler manifolds.

\end{proposition}

\begin{proof}

The idea we use of reducing to the projective case is due to Sz\'ekelyhidi \cite[Proposition 35]{gabor-blowups2}. Indeed, when viewing these invariants as integrals rather than algebro-geometric invariants through Proposition \ref{donaldson-ag-dg}, all calculations are local around the exceptional divisor, and our definition of the metric $\omega_{\epsilon}$ is that it is a glued-in copy of the Burns-Simanca metric on $\Bl_0 \C^n$. Since we use the same metric around in the exceptional divisor in both the projective and non-projective settings, the formulae in the projective case imply those in the general case.  \end{proof}

This allows the calculation of the Futaki invariant on blowups, which is straightforward from Proposition \ref{prop:algebraic-expansions}.
 
\begin{corollary}\cite[Corollary 29]{gabor-blowups1}\label{futaki-expansion} The Futaki invariant is a quotient of polynomials in $\epsilon$ which has the following expansion.

\begin{enumerate}
\item In general, we have$$F_{\epsilon}(u) = F(u) - \frac{\epsilon^{2n-2}}{2(n-2)!}h(p) - \frac{\epsilon^n}{n!}\left(\frac{2n-4}{2}\Delta h(p) - \frac{a_1}{a_0}h(p) \right) + O(\epsilon^{n+1}).$$
\item Suppose in addition $n=2$ and $a_1 \neq 0$. Then $$F_{\epsilon}(u) = F(u) - \frac{\epsilon^2}{2}h(p) + \frac{\epsilon^4a_1}{2a_0}h(p) + \frac{\epsilon^6}{2a_0}\left(\frac{a_1}{3}\Delta h(p) - \frac{h(p)}{2}\right) + O(\epsilon^8).$$
\item Suppose  $n=2$ and $a_1 = 0$. Then $$F_{\epsilon}(u) = F(u) - \frac{\epsilon^2}{2}h(p) - \frac{\epsilon^6}{4a_0}h(p)+ \frac{\epsilon^8}{12a_0} \Delta h(p) + O(\epsilon^{10}).$$
\end{enumerate}Moreover, if $h(p) = \Delta h(p)=0$, then $F_{\epsilon}(u)$ vanishes identically.
\end{corollary}

Note that while the expansion in $\epsilon$ is not actually finite, this is only caused by the fact that the Futaki invariant is given by a \emph{quotient} of polynomials in $\epsilon$. 

One can similarly expand the inner product $\langle \cdot, \cdot\rangle_{\epsilon}$ through Corollary \ref{cor:ip-expansion} and Proposition \ref{prop:algebraic-expansions}, though it does not seem illuminating to explicitly write the resulting inner product. What is important is that, returning to our notation that $(X_0,\alpha_0)$ is the central fibre on which $p$ degenerates to a fixed point $p_0$ of the $T^{\C}$-action on $X_0$, we have a sequence of inner products $\langle \cdot, \cdot\rangle_{\epsilon}$ on the Lie algebra $\mft$. We emphasise that while the $\langle \cdot, \cdot\rangle_{\epsilon}$ are \emph{not} defined purely in terms of invariants of the vector field and Hamiltonian at the fixed point $p_0$, the manner in which  $\langle \cdot, \cdot \rangle_{\epsilon}$ differs from $\langle \cdot, \cdot\rangle_{0}$ \emph{is} purely through invariants of the vector field and Hamiltonian at the fixed point $p_0$. The expansion of $\langle \cdot, \cdot\rangle_{\epsilon}$ in $\epsilon$ is again not finite, but this is simply because it is also a quotient of polynomials in $\epsilon$.

These results allow us to explicitly calculate the extremal vector field $\xi_{\epsilon}$ and hence the term  $$\int_{\Bl_qX}\langle \mu_{\epsilon} , w \rangle \langle \mu_{\epsilon}, \xi_{\epsilon}\rangle  \omega_{\epsilon,q}^n$$ appearing in the relative Donaldson--Futaki invariant, where this is calculated on (the blowup of) the specialisation $q$ of $p$ under a $\C^*$-action. We can define then an inner product $\langle \cdot ,\cdot  \rangle_{\epsilon,q}$ that depends on \emph{both} $\epsilon$ and $q$, which is the Futaki--Mabuchi inner product on $(\Bl_qX,\alpha_{\epsilon})$, and can modify $w$ to make it orthogonal to $\xi_{\epsilon}$ using this inner-product. One similarly checks that one can actually ensure $u$ is orthogonal to the entire Lie algebra $\mft$, matching the statement of the introduction, without changing the relative Donaldson--Futaki invariant. Indeed, one may ensure $u$ is orthogonal by adding an element $u'\in\mft'$ to $u$; this leaves the space underlying the test configuration unchanged as $u\in\mfk^T$, and so the claim follows by checking that the resulting numerical invariants are unchanged.

Summarising, what we have proven is the following:

\begin{corollary}\label{k-stabhence} There is an $\epsilon_0$ such that if the blowup $(\Bl_pX, \alpha_{\epsilon})$ is relatively K-stable for all $0<\epsilon<\epsilon_0$, then for all $u$ orthogonal to $\mft$ under $\langle \cdot ,\cdot  \rangle_{\epsilon,q}$ we have $$A_{\epsilon}h(q_u) + B_{\epsilon}\Delta h(q_u)>0$$ for $\epsilon$ sufficiently small, where $A_{\epsilon}>0$ and $B_{\epsilon}$ depend only on $\epsilon$ and topological invariants of $(X,\alpha)$, and these quantities are all calculated on $\X_0$. Here $q_u$ is the specialisation of $p$ under the flow of $u$.
\end{corollary}

Here $A_{\epsilon}, B_{\epsilon}$  can be computed explicitly from Corollary \ref{futaki-expansion}, with $B_{\epsilon}$ of strictly higher order in $\epsilon$, while the inner product $\langle \cdot, \cdot\rangle_{\epsilon}$ has a similarly explicit interpretation from Corollary \ref{cor:ip-expansion}. Similarly, the numerical criteria produced from GIT can be understood through the existence of zeroes of finite-dimensional moment maps, which is the way in which the original Arezzo--Pacard results were phrased \cite{arezzo-pacard1,arezzo-pacard2}. If $X$ is projective and $\alpha = c_1(L)$ for $L$ ample, the criteria are further phrased in terms of completely classical GIT stability.

\section{The semistable case}\label{sec:semistable}
\subsection{The geometric setup}

We next consider the strictly semistable case, for which we need to assume projectivity: we assume $\alpha = c_1(L)$ for an ample line bundle $L$, so that $X$ is a smooth projective variety. The projectivity argument will be required to appeal to certain results requiring a form of compactness, see Remark \ref{rmk:proj}. Denote by $\xi$ the extremal vector field of $(X,L)$.

\begin{definition} We say that $(X,L)$ is \emph{relatively K-semistable} if there exists a $\xi$-invariant test configuration for $(X,L)$ with central fibre $(X_0,L_0)$, with $c_1(L_0)$ admiting an extremal metric with extremal vector field $\xi$. \end{definition}

The definition implies relative K-semistability in the usual sense. We remark that one may make a more general definition than this, allowing for families whose base is not necessarily $\C$, but it can be seen through a Luna slice argument that the resulting definition is equivalent. 

We will employ the $\xi$-invariant  Kuranishi space $B$ of $(X_0,L_0)$, which is a complex space $B$ with a universal family $\pi: (\Y,\L) \to B$ (which is a holomorphic submersion with $\L$ relatively ample) such that a maximal compact group $K \subset \Aut_{\xi}(X_0,L_0)$ acts $\pi$-equivariantly on $\Y$ and $B$, such that the fibre over $0\in B$ is $(X_0,L_0)$ and by versality of the Kuranishi space, there exists a sequence of points $t_j$ tending to $0$ such that the associated fibres of $\pi$ are isomorphic to $(X,L)$. By construction, $B$ is a complex subspace of a vector space, with the vector space admitting a linear $K$-action. We refer to Sz\'ekelyhidi \cite[Proposition 7]{gabor-deformations} and Inoue \cite[Proposition 3.7]{inoue-moduli} for further details. We will further assume the $\xi$-invariant deformation theory of $(X_0,L_0)$ is unobstructed, which means that the $\xi$-invariant Kuranishi space $B$ is smooth, in order to perform our analytic arguments (this is automatic if the non-equivariant deformation theory of $(X_0,L_0)$ is unobstructed, for example).

\begin{remark}\label{rmk:stabiliser-failure} For a point $b\in B$ with fibre $(\X_b,\L_b)$, it may not be the case that $K_b$ is a maximal compact subgroup of $\Aut(\X_b,\L_b)$, and so the inclusion $K_b^{\C} \subset \Aut(\X_b,\L_b)$ may be strict. As a hypothesis, we will assume this is so.
\end{remark}

Consider the fibre product $\Y \times_B \Y \to B$ (which is a complex manifold by smoothness of $B$), and the diagonal $\Delta \subset \Y \times_B \Y.$ Thus as a set $$\Y \times_B \Y = \{(y_1,y_2) \in \Y \times \Y: \pi(y_1)=\pi(y_2)\},$$ while $$\Delta = \{(y,y) \in \Y \times \Y\} \cong \Y.$$ The blowup $$\X = \Bl_{\Delta} (\Y\times_B\Y)$$ can be thought of as a universal blowup of fibres of $\pi$ at points; we obtain a proper holomorphic submersion $\pi: \X \to \Y$ such that the fibre over $p \in \Y$ is isomorphic to $\Bl_p\Y_{\pi(p)}$, where $\Y_{\pi(p)}$ is the fibre of $\pi$ over $p$. In addition we obtain a sequence of relatively ample line bundles $\L_{\epsilon} = \pi^*\L - \epsilon \E$, where $\E$ is the exceptional divisor of the blowup. The holomorphic submersion $(\X,\L_{\epsilon}) \to \Y$ is then a proper holomorphic submersion whose geometry we will study in detail. 

Much as in the extremal case of Section \ref{sec4}, we will be interested in blowing up a point $p\in \Y$ and will fix a maximal torus $T\subset K_p$ of the stabiliser of $p$ under the $K$-action, and as in the prior section we may assume that $\xi \in \Lie T$, as otherwise the blowup cannot admit an extremal metric, by a calculation similar to Sz\'ekelyhidi \cite[Proposition 40]{gabor-blowups2}. We emphasise that in general $T$ may not be a maximal torus of $\Aut_0(X,L)_p$ (hence of $\Aut_0(\Bl_p,L_{\epsilon})$); as a hypothesis, we will assume throughout that the inclusion $$T \subset \Aut_0(X,L)_p$$ induces a maximal torus. This is automatic, for example, when $\Aut_0(X,L)_p$ (or even $\Aut_0(X,L)$), is actually discrete, which will be the case in our examples. 

We next include metric information. We will fix an initial closed $(1,1)$-form $\Omega \in c_1(\L)$ such that the associated Weil--Petersson form $\Omega_B$ on $B$ (defined by the fibre integral as in Section \ref{moment-map-geometry}) is K\"ahler; this is provided by the Kuranishi theory and work of Ortu described below. We will construct then a sequence $\Omega_{\epsilon} \in c_1(\L_{\epsilon})$ of relative K\"ahler metrics over $\Y$, endowing $\Y$ with a sequence of forms $\Omega_{\Y,\epsilon}$, defined as the Weil--Petersson forms for the submersion $\X\to \Y$ associated to $\Omega_{\epsilon}$. These sequences are built from  the initial $\Omega \in c_1(\L)$, generalising the gluing arguments of Section \ref{sec:analysis} to a family of manifolds over $B$; the construction will also involve perturbing the initial metric $\Omega$ before gluing.

\subsection{The gluing argument}\label{sec:semistable-gluing}

The goal of this section is to construct the K\"ahler metrics $\Omega_{\epsilon}$. The perturbation for $\epsilon=0$ is Ortu \cite[Theorem 2.10]{ortu}, and we begin by describing this. 

We can consider the restriction to $\Y \times_B \Y$ of the projections to the two factors in $\Y \times \Y$, which then are holomorphic maps since $\Y \times_B \Y \subset \Y \times \Y$ is a complex submanifold. The projection to the first factor $\Y$ means that $\Y \times_B \Y \to B$ factors through a map
$$
\Y \times_B \Y \to \Y \to B,
$$
and it is $\Y \times_B \Y \to \Y$ that we consider when $\epsilon = 0$. We will construct a relatively K\"ahler metric on this fibration by pulling back from the projection to the second factor. Unlike the stable case this pullback is not purely vertical as we take the fibre product over $B$ instead of the product. Note also that when we pull back a relatively K\"ahler metric on $\Y \to B$ to $\Y \times_B \Y$ via the projection to the second factor, we obtain a relatively K\"ahler metric for the fibration $\Y \times_B \Y \to \Y$ where the map is the projection to the first factor.  

The form we pull back from the projection to the second factor is the one given by the work of Ortu. By Kuranishi theory, we may consider  $\omega$ as a relatively K\"ahler metric on $\Y \to B$ (as the construction of the Kuranishi space produces $\Y$ as a smooth product of $B$ with the smooth manifold underlying $X_0$). For a function $\phi_b$ on the fibre $\Y_b$ over $b$, denote 
$$
\bar \mfk^T_{\phi_b} = \left\{ h_{u} + \frac{1}{2} \langle u, \nabla \phi_b \rangle : u \in \mfk^T \right\},
$$
where gradient is taken on the fibre $\Y_b$. Ortu \cite[Theorem 2.10]{ortu} proves that, possibly after shrinking $B$, there exists $\phi : \Y \to \mathbb{R}$ such that for all $b \in B$, the restriction $\phi_b$ of $\phi$ is K\"ahler on $\Y_b$, and 
\begin{equation}\label{eq:ortu}
S(\omega + \ddb \phi_b, \Y_b) \in \bar \mfk^T_{\phi_b}.
\end{equation}
We then set $\Omega$ to be the relatively K\"ahler metric on $\Y$, and note that by Kuranishi theory, perhaps after shrinking $B$ once more, the associated Weil--Petersson metric on $B$ is K\"ahler. We then pull back this $\Omega$ to  $\Y \times_B \Y \to \Y$, via the projection to the second factor, producing a relatively K\"ahler metric on $\Y \times_B \Y \to \Y$. We will next modify this form to a sequence $\Omega_{\epsilon}$ on $\X$ by a gluing argument, in such a way that $\Omega_0 = \Omega$, by proving an analogue of Theorem \ref{thm:fibrewisesoln} for this family $\X\to\Y$.

We follow the steps in the extremal case. First, let $d$ be the distance function on $\Y \times_B \Y$ to the diagonal $\Delta$ computed on the fibre. Define the first approximate solution as 
$$
\Omega^1_{\epsilon} = \sigma^* \Omega_{0} + \epsilon^2 \ddb\left(\gamma_2 \cdot (\gamma(\epsilon^{-1}d)\log (\epsilon^{-2}d^2) + f(\epsilon^{-2}d^2))\right),
$$
where $\sigma : \X \to \Y \times_B \Y$ denotes the blowdown map. This may not be relatively K\"ahler over all of $\Y$, but we can certainly guarantee this for all sufficiently small $\epsilon$ after shrinking $B$. All holomorphic vector fields on $\Y$ lift to $\X$ as the diagonal in $\Y \times_B \Y$ is preserved by the product $K^T$-action, and we use the notation $h_{\epsilon}$ for the lifted potential for the vector field induced by $h \in \mfk$ with respect to $\Omega^1_{\epsilon}$. Note that as we changed the relative K\"ahler metric $\omega$ to $\Omega_0$ before blowing up, the holomorphy potentials also changed, and so the expansion of the potential $h_{\epsilon}$ associated to some $u \in \mfk^T$ is  now
$$
h_{\epsilon} = h + \frac{1}{2} u (\phi) + \gamma_2 \cdot O(1)
$$
which is $O(|b|)+\gamma_2 \cdot O(1)$, since $\phi$ is $O(|b|)$. We also define weighted spaces on the fibres of $\X \to \Y$ analogously to how they were defined on the central fibre in Section \ref{sec:wtdspaces}---this used holomorphic normal coordinates about the blown up point, but it is equivalent to use the function $d$ to define these. It is then clear that the definition goes over also outside the central fibre in our current setup.

After possibly shrinking $\Y$ again, we can then ensure the analogue of Proposition \ref{prop:linearisationinverse} holds on every fibre of $\X^T \to \Y^T$. For $y \in \Y^T$, let $\xi_y \in \mfk^T$ be the vector field which $S_V(\Omega_0)$ corresponds to on the image of $y$ in $B$ via \eqref{eq:ortu}. Note that $\xi_y$ only depends on the image of $y$ in $B$.
\begin{proposition}
\label{prop:linearisationinversesemistable}
Let $n>2$ and let $\delta \in (4-2n, 0).$ Then, after possibly shrinking $Y$, we have that for all $y \in \Y^T$, the operator 
$$
P : C^{4,\alpha}_{\delta} (\X^T_y) \times \bar \mfk^T \to C_{\delta-4}^{0,\alpha} (\X^T_y)
$$
given by
$$
(f,h) \mapsto L_{\omega_{\epsilon,y}}(f) - \frac{1}{2} \langle \xi_y, \nabla f \rangle - h_{\epsilon} 
$$
where $\omega_{\epsilon,y}$ is the restricton of $\Omega^1_{\epsilon}$ to $\X_b$, admits a right-inverse $Q$ with operator norm bounded independently of $\epsilon$. Moreover, $Q$ depends smoothly on $y$.

In the case when $n=2$, the same holds for $\delta \in (-1,0)$ for all $|\delta|$ sufficiently small, but with $\| Q \| \leq C \epsilon^{\delta}$.
\end{proposition}
 The key is that since we can find a right-inverse on the central fibre, we can find right-inverses for all fibres near the central fibre, with a uniform bound for their operator norms. This is part of Ortu's proof establishing Equation \eqref{eq:ortu}, see \cite[Equation 2.6]{ortu}. Therefore we uniformly obtain right-inverses with the correct properties before blowing up. The proof then proceeds exactly as the proof of Proposition \ref{prop:linearisationinverse}. In the case $n=2$, the proof goes by establishing a uniform bound (with no $\epsilon$ dependence) on the subspace functions of average $0$ first -- we can then do this uniformly over all of $\Y$ (after perhaps shrinking $B$ independently of $\epsilon$). The dependence on $\epsilon$ then comes from going from this operator to the one where we also consider functions that are not of average $0$.

With this in place, we can improve the approximate solution as in Section \ref{sec:improving}. For every $y \in \Y^T$, there is a 
$$
\Gamma_y = - d^{4-2n} + \widetilde \Gamma_y,
$$
where $\widetilde{\Gamma}_y$ is $O(d^{5-2n})$ such that on each fibre, we have
$$
\mathcal{D}^* \mathcal{D} \Gamma_y = c \mu (y) + \frac{c}{\Vol(X)} - c \delta_{y}.
$$
We then modify $\Omega^1_{\epsilon}$ by the function $\Gamma$ whose restriction to $\pi^{-1}(y)$ is $\Gamma_y$ by letting
$$
\Omega^2_{\epsilon} = \Omega^1_{\epsilon} + \ddb\left(\epsilon^{2n-2} \gamma_1  \Gamma \right).
$$

Letting $\Omega^i_{\epsilon,y}$ denote the restriction of $\Omega_{\epsilon}$ to $\X_y$ we have
\begin{align*}
S(\Omega^2_{\epsilon, y}) =& S(\Omega^1_{\epsilon, y}) + L_{\Omega^1_{\epsilon, y}} (\epsilon^{2n-2} \gamma_1  \Gamma) + R_{\Omega^1_{\epsilon,y}}(\epsilon^{2n-2} \gamma_1  \Gamma)  \\
=&S(\Omega^1_{\epsilon, y}) + L_{\Omega^1_{\epsilon, y}} (\epsilon^{2n-2} \gamma_1  \Gamma) + O(\epsilon^{2n-1}),
\end{align*}
where $R_{\omega_{\epsilon,p}}$ is the non-linear part of the scalar curvature operator.

Again this is an improved approximate solution to the fibrewise equation. Let $h'_{\epsilon,y}$ denote the potential with respect to $\Omega^1_{\epsilon}$ associated to $S_V(\Omega_0) + \epsilon^{2n-2} h_y \in \bar \mfk^T$, i.e. the potential corresponding to $\xi_y + \epsilon^{2n-2} \mu (y) \in \mfk^T$, and let $\xi'_{\epsilon,y}$ denote the corresponding real holomorphic vector field. We then obtain the analogue of Lemma \ref{lem:approxsoln}, that is, after possibly shrinking $\Y$,  for all $0 < \epsilon \ll 1$ and for all $y \in \Y$,
\begin{align*}
\left\| S(\Omega^2_{\epsilon, y}) - \frac{1}{2} \xi'_{\epsilon,y}(\epsilon^{2n-2} \gamma_1  \Gamma_y) - h'_{\epsilon, y} \right\|_{C^{0,\alpha}_{\delta-4}} \leq C r_{\epsilon}^{4-\delta}.
\end{align*}
with $\delta$ chosen as in Lemma \ref{lem:approxsoln}.

We are now ready to perturb and obtain the main result of the section. As $\pi: \X^T \to \Y^T$ is a $K^T$-equivariant holomorphic submersion, we can define function spaces 
$\bar \mfk^T_{\pi,\epsilon}$ precisely as in the previous cases. These depend on a relatively K\"ahler metric on $\X^T$ that we write as $\Omega^1_{\epsilon} + \ddb \phi_{\epsilon}$ for some $\phi_{\epsilon}$. The space then consists of the functions on $\X^T$ whose restriction to a fibre equals the restriction of some $h_{\epsilon} + u (\phi_{\epsilon})$ to the fibre, where $h_{\epsilon}$ is the potential of $u \in \mfk^T$ with respect to $\Omega^1_{\epsilon}.$ Note that $u$ can depend on $y$ and $\epsilon$.

We in addition define function spaces $\bar \mfk^T_{V,\epsilon}$ and what we actually solve is that $S_V(\Omega_{\epsilon}) \in \bar \mfk^T_{V,\epsilon}$. This is analytically more straightforward than employing $\bar \mfk^T_{\pi,\epsilon}$ directly. Writing $\Omega_{\epsilon}$ as $\Omega_{\epsilon}^1+\ddb\phi_{\epsilon}$ and letting $\phi_{\epsilon,y}$ denote the restriction of $\phi_{\epsilon}$ to $\X_y$, $\bar\mfk^T_{V,\epsilon}$ consists of the functions on $\X$ whose restriction to a fibre $\X_y$ equals 
$$
h_{\epsilon} + \frac{1}{2} \langle u, \nabla \phi_{\epsilon, y} \rangle
$$
for some $u \in \mfk^T$ that can depend on $y$ and $\epsilon$. As above, $h_{\epsilon}$ is the potential of $u$ with respect to $\Omega^1_{\epsilon}$ restricted to $\X_y$. We use the relatively K\"ahler metric $\Omega^1_{\epsilon}$ to compute the above gradient and inner product on the fibre $\Y_y$.

The main result is then the following.
\begin{theorem}
\label{thm:semistablefibrewisesoln}
Suppose $n>2$. 
After possibly shrinking $B$, 
we have that for all $0 < \epsilon \ll 1$ and for all $y \in \Y^T$  
there exists $\phi_{\epsilon,y} \in C^{\infty}(\X^T_y)$ and $h_{\epsilon, y}$ corresponding to a vector field $u_{\epsilon,y} \in \mfk^T$ such that
$$
S(\Omega_{\epsilon,y}+\ddb \phi_{\epsilon,y}) - \frac{1}{2} \langle u_{\epsilon,y}, \nabla(\phi_{\epsilon,y}) \rangle - h_{\epsilon,y} = 0.
$$
The potential $ \frac{1}{2} \langle u_{\epsilon,y}, \nabla(\phi_{\epsilon,y}) \rangle + h_{\epsilon,y}$ admits an expansion
$$
 \frac{1}{2} \langle u_{\epsilon,y}, \nabla(\phi_{\epsilon,y}) \rangle+ h_{\epsilon,y}= S_V(\Omega) + \epsilon^{2n-2} \left(c \mu(y) + \frac{c}{\Vol(X)} \right) + O(\epsilon^{2n-1}) + \gamma_2 O(\epsilon^2),
$$
where the $O(\epsilon^{2n-1})$-term is over the whole of $\X_y$ and the term $\gamma_2 O(\epsilon^2)$ is an $O(\epsilon^2)$-function supported on $B_{2r_{\epsilon}}$. 

In the case $n=2$, the same holds except 
that in the expansion of the potential the $O(\epsilon^{2n-1})$-term is $O(\epsilon^{2n-2+\theta})$ for some $\theta >0$.

Finally, the function $\phi_{\epsilon}$ on $\X^T$ whose restriction to the fibre $\X_y$ is $\phi_{\epsilon, y}$ is smooth. Thus if we write $\Omega_{\epsilon} = \Omega^1_{\epsilon} + \ddb \phi_{\epsilon}$, then $S_V(\Omega_{\epsilon}) \in \bar \mfk^T_{V,\epsilon}$.
\end{theorem}

\subsection{The Kempf--Ness argument}\label{semistableKN} Having established the main analytic results in this setting, we turn to the geometry. Consider the forms $\Omega_{\epsilon} \in c_1(\L_{\epsilon})$ which are relatively K\"ahler on the submersion  $\X^T \to \Y^T$, and the associated Weil--Petersson form $\Omega_{\Y,\epsilon}$ on $\Y^T$ produced through a fibre integral over $\X^T\to \Y^T$. 

Let $\mu_{\epsilon}: \X \to (\mfk^T)^*$ be the moment map for the $K^T$-action on $(\X,\Omega_{\epsilon})$. From the setup, we obtain a moment map $\sigma_{\epsilon}: \Y^T \to (\mfk^T)^*$ for the $K_T$-action on $(\Y, \Omega_{\Y,\epsilon})$, which takes the form $$\langle \sigma_{\epsilon},u\rangle(p) = \int_{\Bl_{p}\Y_{\pi(p)}}\langle \mu_{\epsilon},u\rangle (S_V(\Omega_{\epsilon}) - \hat S_V) \Omega_{\epsilon}^n.$$ We have two relevant function spaces; the first is $\bar\mfk^T_{\epsilon, \pi,p}$ which is defined to be the span of the $\langle \mu_{\epsilon},u\rangle|_{\Bl_{p}\Y_{\pi(p)}}$, while the second $\bar\mfk^T_{\epsilon,V,p}$ satisfies $$S_V(\Omega_{\epsilon}) - \hat S_V \in \bar\mfk^T_{\epsilon, V,p}$$ by Theorem \ref{thm:semistablefibrewisesoln}.

We fix a point $p\in X^T$, which we view as a point in the corresponding fibre of $\X^T\to \Y^T$. We then obtain a unique extremal vector field $\xi_{\epsilon} \in \mft$ on the blowup $(\Bl_pX,L_{\epsilon}$ which is independent of point in the orbit $T^{\C}.p$ (this is where we use that by hypothesis $T$ induces a maximal torus in $\Aut_0(X,L)_p$. As before, we may further expand this extremal vector field in powers of $\epsilon$. We begin with the following analogues of results proven in the extremal case in Section \ref{moment-map-geometry}.

\begin{proposition} The following hold:
\begin{enumerate}
\item  A zero of the moment map $\sigma_{\epsilon}$ is an extremal metric;
\item The moment map is a moment map with respect to a K\"ahler metric.
\end{enumerate}
 \end{proposition}
 
 These are proven similarly to the extremal case, so we omit the proofs. The key new point in the first part, for example, is the presence of an $O(|b|)$-term (measuring the variation of $\Omega$ on the fibres $\Y_b$) that we can ensure is as small as we would like if we shrink $B$. Indeed, the projection operator as considered in  Lemma \ref{OS-lemma}  takes the form
\begin{align*}
\sum_{i,j=1}^k a_i \frac{\left(\int_{\X_y} \left(h_{\epsilon,i} + u_i(\phi_{\epsilon,y})\right)\Omega_{\epsilon,y}^n \right) \left( \int_{\X_y} (h_{\epsilon,j} + u_j(\phi_{\epsilon})) \Omega_{\epsilon,y}^n \right)}{\left( \int_{\X_y}(h_{\epsilon,j} + u_j(\phi_{\epsilon}))^2 \Omega_{\epsilon,y}^n \right)^{\frac{1}{2}}}
\end{align*}
and as we have remarked in the previous section, the $h_{\epsilon,i}$ now expand in the same way as in Lemma \ref{OS-lemma}, up to an $O(|b|)$-term. Thus when $|b|$ is sufficiently small, the projection is an isomorphism. So after possibly shrinking $\Y$ yet again, we get that a zero of the moment map is equivalent to the metric on the fibre being extremal. 

For the second part, the key point is that the moment maps satisfy
\begin{align*}
\langle \sigma_{\epsilon},u \rangle(p) =& \int_{\Y_{\pi(p)}}\langle \mu_{0},w\rangle (S_V(\Omega_0) - \langle\mu_0, \xi_{\epsilon}\rangle - \hat S_{V,0}) \Omega_{0}^n \\
&+ \epsilon^{2n-2}( \langle \mu_0,u\rangle(p) - \langle \xi',u\rangle_0)+O(\epsilon^{2n-2+\theta}),
\end{align*} 
where now $S_V(\Omega_0) -\langle\mu_0, \xi_{\epsilon}\rangle-  \hat S_{V,0}$ is nonconstant as we do not begin with an extremal metric. Note that $\xi'$ is still the $O(\epsilon^{2n-2})$-term in the expansion of the extremal vector field---this is constant as we have assumed the torus $T$ is a maximal torus in $\Aut_0(X, \alpha)_p$. Now, $\int_{\Y_{\pi(p)}}\langle \mu_{0},w\rangle (S_V(\Omega_0) - \langle\mu_0, \xi_{\epsilon}\rangle - \hat S_{V,0}) \Omega_{0}^n$ is a moment map for the pullback to $\Y$ of the form $\Omega_B$ on $B$ (and in fact, the $\epsilon^0$-term of $\Omega_{\Y, \epsilon}$ \emph{is} $\Omega_B$). The replacement K\"ahler metric on $\Y$ then takes the form, for $\theta>0$ (and $\theta=1$ when $n\geq 3$) $$\Omega_B + \epsilon^{2n-2}\Omega_{0} + O(\epsilon^{2n-2+\theta}).$$ The proof that we can replace the actual form with this is as in Lemma \ref{lem:formexp}, where we see that the key new difference is the form of the expansion at $0$. Note that $\mu_0$ is the moment map on $\Y$, and this is why we see the term $\Omega_0$ at order $\epsilon^{2n-2}$ above.

  Thus our geometry has returned us to the situation where we wish to find a zero of the moment map in the given $G^T$-orbit of $p$.  The first issue we must overcome is that the manifold $\Y$ is not compact, and to appeal to a version of the Kempf--Ness theorem we will require compactness. Here we use our projectivity hypothesis---namely that $\L$ is relatively ample---to embed $\Y \to B$ into a product of projective spaces, $K$-equivariantly. To do so, we consider the sequence $\pi_*\L^{\otimes r}$ of pushforwards; by relative ampleness, these are vector bundles for $r\gg 0$ admitting a $K$-action.  
\begin{lemma}
The vector bundle $\pi_*(\L^{\otimes r})$ may be equivariantly trivialised in a neighbourhood of $0 \in B$.
\end{lemma}

\begin{proof}
We expect this to be well-known, so we merely sketch the proof, following the approach of Segal in the smooth setting \cite[Section 1]{segal}. Consider the trivial bundle $$B\times \pi_*(\L^{\otimes r})_0 = B\times H^0(\Y_0,r\L_0)$$ over $B$; note that $0\in B$ is a fixed point of the $K$-action, so that this trivial bundle is also admits a $K$-action. The restriction if this trivial bundle to $0\in B$ is $K$-equivariantly isomorphic to the restriction of $\pi_*(\L^{\otimes r})$ over $0$ (by definition), giving a $K$-invariant section $e_0$ of the $\Hom$-bundle $$ \Hom(\pi_*(\L^{\otimes r}), B\times H^0(\Y_0,\L^{\otimes r}_0))|_0.$$ We may extend $e_0$ to a holomorphic section of the bundle $\Hom(\pi_*(\L^{\otimes r}), B\times H^0(\Y_0,\L^{\otimes r}_0))$ by trivialising, and may further extend it in a $K$-invariant manner by averaging over the compact Lie group $K$ (through the Haar measure), producing a holomorphic $K$-invariant section $e$ of $\Hom(\pi_*(\L^{\otimes r}), B\times H^0(\Y_0,\L^{\otimes r}_0))$. The section $e$ is invertible at $0\in B$, and as this is an open condition, it is invertible in a neighbourhood of $0$, producing an equivariant trivialisation of the bundle $\pi_*(\L^{\otimes r})$. \end{proof}

By choosing such an equivariant trivialisation, we may $K$-equivariantly embed $\Y$  into $B\times \pr(H^0(X_0,rL_0))$. Since $B$ is a subspace of a vector space, we may further embed $B$ into a projective space in a $K$-equivariant manner, by linearity of the $K$-action. Thus produces a $K$-equivariant embedding of $\Y$ into a product of projective spaces, and in turn we may use the Segre embedding to equivariantly embed $\Y$ as a submanifold of a single projective space $\pr^N$.

\begin{corollary} For any $\epsilon$ with $\Omega_{\epsilon}$ K\"ahler, perhaps after shrinking $B^T$ and restricting $\Y^T$, there is a $K$-invariant K\"ahler metric $\Omega_{\pr^N, \epsilon}$ on $\pr^N$ and a $K^T$-equivariant holomorphic embedding $\Phi: \Y^T \to \pr^N$ such that $\Phi^*\Omega_{\pr^N} = \Omega_{\epsilon}$. Furthermore, there is a moment map for the $K^T$-action on $(\pr^N,\Omega_{\pr^N, \epsilon})$ whose restriction to $\Y^T$ is $\sigma_{\epsilon}$.
\end{corollary}

\begin{proof}
We have already constructed the embedding $\Phi$, so we merely wish to extend the K\"ahler metric $\Omega_{\epsilon} $ from $\Y^T$ to $\pr^N$ in a $K^T$-invariant manner, perhaps after shrinking $B^T$. The analogous extension result is well-known without  assuming $K^T$-invariance, see for example Coman--Guedj--Zeriahi \cite[Proposition 2.1]{CGZ}. Averaging the resulting extension over $K^T$ produces a $K^T$-invariant extension. Extending the moment map is then standard by uniqueness results for moment maps, see e.g. \cite[Lemma 4.7]{DMS}. \end{proof}

We now wish to appeal to a version of the Kempf--Ness theorem. As we are interested in a local version of this result, involving  information only in $\Y^T$ (which is noncompact) rather than the overlying projective space, we employ the gradient flow approach to the Kempf--Ness theorem. This gradient flow is the flow $$\frac{d}{dt}p(t) = -Jv_{\sigma_{\epsilon}(p(t))},$$ where $\sigma_{\epsilon}$ is the moment map, $J$ is the almost complex structure on $\Y^T$ and $v_{\sigma_{\epsilon}(p(t))} \in T_{p(t)}\Y^T$ is the associated tangent vector. This is identical to the corresponding flow on $\pr^N$.

We next consider the corresponding geometry for $\epsilon=0$, namely the flow on $(B,\omega_B)$, starting at $b = \pi(p)$. By Ortu \cite[Proposition 2.4]{ortu}, perhaps after shrinking $B$, we may assume that the flow converges to a zero of the moment map $b_{\infty} \in \overline{G.b} \cap B$, which is an extremal metric (see \cite[Theorem A.1]{ortu} for related results). In the cscK case, it then follows from Chen--Sun \cite{chen-sun} that $b_{\infty}=0 \in B$, so that $0\in \overline{G.b}$; this sort of result will be important in controlling the flow for $\epsilon>0$. In general, we wish to replace $B$ with a $B'$ such that the origin is in the closure of the orbit of a point corresponding to $b$. 

To be more precise, the point $b_{\infty}$ is GIT polystable in $\pr^N$ (as it is a zero of the moment map), so we appeal to the Luna slice theorem \cite{luna} (see also Wang \cite[Theorem 2.5]{wang} for a more analytic perspective). This produces a complex manifold $B'$ with a $K^T_{b_{\infty}}$-action and a $K^T_{b_{\infty}}$-equivariant holomorphic map $\Phi: B'\to B$, such that by construction $B'$ is a submanifold of a vector space with a linear $K^T_{b_{\infty}}$-action, $\Phi(0) = b_{\infty}$ and crucially there is a point $b' \in B'$ with $\Phi(b') = b$ and $0 \in \overline{G^T_{b_{\infty}}.b'},$ where $G^T_{b_{\infty}} = (K^T_{b_{\infty}})^{\C}$ is reductive (as $b_{\infty}$ is polystable). We may pull back the universal family $\Y$ to $B'$ to produce a $K^T_{b_{\infty}}$-equivariant family $\Y'\to B'$. We may further assume, by the same analytic arguments as before, that for this new family, the fibrewise scalar curvature lies in the associated function space isomorphic to trivial bundle with fibre $\Lie K_{b_{\infty}}$ over all of $(\Y')^T$. To simplify notation, we simply replace $B$ with $B'$ and use the notation $\Y\to B$, where now $0\in \overline{G^T.b}$. Note that $0$ is the only fixed point of $G^T$ in the closure of the orbit of $b'$, since fixed points are polystable, and polystable elements in orbit closures are unique.

We use the conclusion that $0\in \overline{G^T.b}$ to understand the geometry of the moment map flow. We assume that $T$ is actually trivial, as we may reduce to this situation by projecting the moment maps orthogonally to $\mft^*$ if not (and the resulting argument is then identical), as in Ortu \cite[Section 2.1]{ortu}.

\begin{lemma}
\label{lem:flowbound}
For all 
$0<\epsilon\ll 1$, there is a $\delta>0$ such that if $p=p(0)$ satisfies 
$|\pi(p)| \leq \delta$, then $|\pi(p(t))| \leq \delta$ for all $t\geq 0$. 
\end{lemma}

\begin{proof} We first explain the situation when $\epsilon = 0$. Recall that the K\"ahler form we have on $\Y$ is given by 
$$\Omega_B + \epsilon^{2n-2}\Omega_{0} + O(\epsilon^{2n-2+\theta}).$$
Since the map $\Y \to B$ is $K$-equivariant, this implies that the flow when $\epsilon = 0$ is such that $\pi(p(t)) = b(t)$, where $b(t)$ is moment map flow associated to $\Omega_B$ on $B$. In other words, the flow on $\Y$ (or really its compactification) covers the flow on the (compactification) of $B$ when $\epsilon = 0$. 

Moreover, as in \cite[Proposition 4.5]{DMS}, we have the bound 
$$
- \left( \frac{d}{dt} (b(t)) \right). b(t) \geq c |q(b(t))|^4,
$$ 
where $q$ is the projection from the tangent space $T_b B$ to the tangent space of its $G$-orbit and on the left hand side we are taking the Euclidean inner product in the vector space $B$ is a a ball in. Now, as $0$ is the only fixed point of the $K$-action in the closure of the $G$-orbit of $b$ in $B$, we can as in \cite[Proposition 4.5]{DMS} mutually bound $|q(b)|$ and $|b|$ on the annulus $B_{\delta} \setminus B_{\delta/2} \cap \overline{G.b}$. The upshot is that
$$
- \left( \frac{d}{dt} (b(t)) \right). b(t) \geq c |b(t)|^4
$$  
on the annulus $(B_{\delta} \setminus B_{\delta/2}) \cap \overline{G.b}$ in the closure of the orbit of $b$.

The moment map agrees with the above up to a term which is $O(\epsilon^{2n-2})$. In fact, infinitesimally, the change in $\pi(p(t))$ is given by the \emph{horizontal} part of $\frac{d}{dt}(p(t)).$ Moreover, $\Omega_{0}$ equals $\omega$ over the central fibre and equals this up to a term which is $O(|b|)$ in general. This implies that this term is $O(\epsilon^{2n-2}|b|)$, and so
\begin{align*}
- \left( \frac{d}{dt} (\pi(p(t))) \right). \pi(p(t)) \geq& c |\pi(p(t))|^4- C\epsilon^{2n-2} |\pi(p(t))|^3 
\end{align*}
for any $p(t) \in \pi^{-1}(B_{\delta} \setminus B_{\delta/2}) \cap \overline{G.p}$. Thus for any sufficiently small $\epsilon$, 
$$
 \left( \frac{d}{dt} (\pi(p(t))) \right). \pi(p(t)) = \frac{d}{dt} \left( |\pi(p(t))|^2 \right)
 $$
 is \emph{negative} for any $p(t)$ such that $|\pi(p(t))| \leq \delta$. 
 In particular, 
 $|\pi(p(t))| \leq \delta$ all $t$ if this holds at $t=0$, as required.
\end{proof}

 This means that, after shrinking $B$, we may assume that the flow $p(t)$ lies in $\Y$ for all $t\geq 0$. 
 This allows us to apply the following result, which applies as $\Y$ is $K$-equivariantly embedded in a projective space such that the K\"ahler metric and moment map extend.

\begin{theorem}\cite[Corollary 4.14]{DMS} For  all $0<\epsilon  \ll 1$ either the flow converges to a point $p_{\infty} \in G^T.p \cap \Y$ satisfying $\sigma_{\epsilon}(p_{\infty})= 0$, or there is a $\lambda: \C^*\hookrightarrow G^T$ such that $\lim_{t\to 0} \lambda(t).p = q$ with  $q\in \Y$ and with $$\langle \sigma_{\epsilon}(q), v_{\lambda}\rangle <0,$$ where $v_{\lambda} \in \mfk^T$ exponentiates to $\lambda$.
\end{theorem}

\begin{remark}\label{rmk:proj}The difference with the usual Kempf--Ness theorem is that we may conclude $q\in \Y$, which is nontrivial as $\Y$ is noncompact. Appealing to this result is the reason we must assume $X$ is projective in the current section, as it relies on embedding $\Y$ into a projective space.
\end{remark}

\begin{corollary}

If $(\Bl_pX,\alpha_{\epsilon})$ is relatively K-stable with respect to the extremal vector field, then $c_1(L_{\epsilon})$ admits an extremal metric.

\end{corollary}

\begin{proof}

The $\C^*$-action $\lambda$ induces a test configuration in such a way that the value $\langle \sigma_{k,\epsilon}(q), v_{\lambda}\rangle$ agrees with the associated relative Donaldson--Futaki invariant defined in Section \ref{sec:rel-kstable}, and relative K-stability means this must be nonnegative, with equality if and only if $(\Bl_pX,c_1(L_{\epsilon})) \cong (\Bl_{p_{\infty}}X,c_1(L_{\epsilon}))$, which implies the result. Thus if $(\Bl_pX,c_1(L_{\epsilon}))$ is relatively K-stable, then the point $p_{\infty}$ must satisfy  $\sigma_{\epsilon}(p_{\infty})= 0$ and must also satisfy $(\Bl_{p_{\infty}}X,c_1(L_{\epsilon})) \cong (\Bl_pX,c_1(L_{\epsilon}))$. It follows  that $(\Bl_{p_{\infty}}X,c_1(L_{\epsilon}))$ admits an extremal metric as $\sigma_{\epsilon}(p_{\infty})$ vanishes, so since $(\Bl_{p_{\infty}}X,c_1(L_{\epsilon})) \cong (\Bl_pX,c_1(L_{\epsilon}))$, this implies that  $(\Bl_pX,c_1(L_{\epsilon}))$ admits an extremal metric.
\end{proof}

One can use the results of Section \ref{GIT-section} to obtain a more explicit GIT characterisation; we omit the details.

\subsection{Applications} As the results in the K-semistable case are general and quite technical, we end the paper with a concrete application to a K-semistable manifold with $\Aut(X,L)$ is discrete. It follows from the results of Section \ref{GIT-section} that if we blow up a point $p\in X$ which is GIT stable, viewed as a point in $\Y$ (the universal family over the Kuranishi space), then $(\Bl_pX,L_{\epsilon})$ admits a cscK metric. 

\begin{theorem}\label{thm:open}
Suppose that $(\Y_0,\L_0)$ satisfies the condition that $\Aut_0(\Y_0,\L_0) \cong \C^*$. Then the blowup $(\Bl_pX,L_{\epsilon})$ of a general point on $p\in X$ admits cscK metrics for all $0<\epsilon\ll 1$.
\end{theorem}

\begin{proof}

In this case, we may perform the arguments of Sections \ref{sec:semistable-gluing} and \ref{semistableKN} on the test configuration $(\Y,\L)$ for $(X,L)$ with central fibre $(\Y_0,\L_0)$. The claim then follows by choosing a point in $(\Y_0,\L_0)$ and using Zariski openness of the stable locus, meaning that a general point in $X$ is actually GIT stable, implying $c_1(L_{\epsilon})$ admits cscK metrics for all $0<\epsilon\ll 1$.

\end{proof}

\begin{example} Note that Theorem \ref{thm:open} can be applied to give many new examples of manifolds admitting extremal metrics. Indeed, there are now known many explicit examples of strictly K-semistable Fano threefolds, that admit a degeneration to a K-polystable Fano (see \cite{fanothreefolds} and the references therein). In order to apply our construction, the central fibre of such a degeneration needs to be smooth. Theorem \ref{thm:open} then guarantees the existence of a cscK metric provided the reduced automorphism group is $\C^*$. This holds for certain members of the family 1.10 of the Mori--Mukai list of smooth Fano threefolds, which is the family that includes the Mukai--Umemura manifold (however note that the Mukai--Umemura manifold itself has larger reduced automorphism group). One can find other examples for instance in the families 2.20, 2.21, 2.22, 3.5, 3.8, 3.10, 3.12 and  4.13 of the Mori--Mukai list. Some of these are infinite families to which the construction applies.
\end{example}

\vspace{4mm}


\begin{thebibliography}{99}

\bibitem{fanothreefolds}
C. Araujo, A.-M. Castravet, I. Cheltsov, K. Fujita, A.-S.
  Kaloghiros, J. Martinez-Garcia, C. Shramov, H. S\"u{\ss}, and
  N. Viswanathan, \emph{{The Calabi Problem for Fano threefolds}}, MPIM
  Preprint Series 31 (2021).


\bibitem{arezzo-pacard1}
C. Arezzo and F. Pacard,  \emph{Blowing up and desingularizing constant scalar curvature K\"ahler manifolds}. Acta Math. 196 (2006), no. 2, 179--228. 

\bibitem{arezzo-pacard2}
C. Arezzo and F. Pacard, \emph{Blowing up K\"ahler manifolds with constant scalar curvature. II. } Ann. of Math. (2) 170 (2009), no. 2, 685–738. 

\bibitem{APS}
C. Arezzo, F. Pacard and M. Singer, \emph{Extremal metrics on blowups.} Duke Math. J. 157 (2011), no. 1, 1--51. 


\bibitem{chen-sun}
X. Chen and S. Sun, \emph{Calabi flow, geodesic rays, and uniqueness of constant scalar curvature K\"ahler metrics.} Ann. of Math. (2) 180 (2014), no. 2, 407--454.

\bibitem{CGZ}
D. Coman, V. Guedj and A. Zeriahi , \emph{Extension of plurisubharmonic functions with growth control.} J. Reine Angew. Math. 676 (2013), 33-49.

\bibitem{datar} 
V. Datar, \emph{Expansions of solutions to extremal metric type equations on blowups of cscK surfaces.} Ann. Global Anal. Geom. 55 (2019), no. 2, 215--241. 

\bibitem{relative}
R. Dervan, \emph{Relative K-stability for K\"ahler manifolds.} Math. Ann. 372 (2018), no. 3-4, 859--889. 

\bibitem{DH}
R. Dervan and M. Hallam, \emph{The universal structure of moment maps in complex geometry},  	arXiv:2304.01149, 39pp.

\bibitem{DMS}
R. Dervan, J. B. McCarthy, L. M. Sektnan, \emph{$Z$-critical connections and Bridgeland stability conditions.} Camb. J. Math. 12 (2024), no.2, 253--355.

\bibitem{kahler}
R. Dervan and J. Ross, \emph{K-stability for K\"ahler manifolds.} Math. Res. Lett. 24 (2017), no. 3, 689--739. 

\bibitem{donaldson-moment}
S. K. Donaldson, \emph{Remarks on gauge theory, complex geometry and 4-manifold topology.} Fields Medallists' lectures, 384--403,
World Sci. Ser. 20th Century Math., 5, World Sci. Publ., River Edge, NJ, 1997.

\bibitem{donaldson}
S. K. Donaldson, \emph{Scalar curvature and stability of toric varieties. } J. Differential Geom. 62 (2002), no. 2, 289--349.

\bibitem{donaldson-lower}
S. K. Donaldson, \emph{Lower bounds on the Calabi functional.} J. Differential Geom. 70 (2005), no. 3, 453--472.

\bibitem{fujiki}
A. Fujiki, \emph{Moduli space of polarized algebraic manifolds and {K}\"ahler metrics.} Sugaku Expositions Vol. 5 no. 2 (1992) 173--191.

\bibitem{futakimabuchi95}
A. Futaki and T. Mabuchi, \emph{Bilinear forms and extremal {K}\"{a}hler vector fields associated with {K}\"{a}hler classes}. Math. Ann., 301 (1995), no. 2, 199--210.


\bibitem{moment-weight}
V. Georgoulas, J. Robbin, D. A. Salamon, \emph{The Moment-Weight Inequality and the Hilbert–Mumford Criterion}. Lecture Notes in Math., 2297
Springer, Cham, (2021), vii+190 pp.

\bibitem{inoue-moduli}
E. Inoue, \emph{The moduli space of Fano manifolds with K\"ahler-Ricci solitons}. Adv. Math. no. 357 (2019). 1--65.


\bibitem{legendre}
E. Legendre, \emph{Localizing the Donaldson--Futaki invariant}. Internat. J. Math.32 (2021), no.8, Paper No. 2150055, 23pp.

\bibitem{luna}
D. Luna, \emph{Slices \'etales.} Sur les groupes alg\'ebriques,  Suppl\'ement au Bull. Soc. Math. France, Tome 101, Société Mathématique de France, Paris, 1973, . 81--105.

\bibitem{JBM}
John B. McCarthy, \emph{Canonical metrics on holomorphic fibre bundles},  	arXiv:2202.11630, 15pp.

\bibitem{GIT}
D. Mumford, J.  Fogarty and F. Kirwan, \emph{Geometric invariant theory.} Third edition. Ergebnisse der Mathematik und ihrer Grenzgebiete (2) [Results in Mathematics and Related Areas (2)], 34. Springer-Verlag, Berlin, 1994. xiv+292pp. 

\bibitem{ortu}
A. Ortu, \emph{Moment maps and stability of holomorphic submersions},  	arXiv:2407.03246, 27pp.

\bibitem{OS}
A. Ortu and L. M. Sektnan, \emph{Constant scalar curvature K\"ahler metrics and semistable vector bundles}, arXiv:2406.08284, 43pp.

\bibitem{pacard} 
F. Pacard, \emph{Constant scalar curvature and extremal K\"ahler metrics on blow ups.} Proceedings of the International Congress of Mathematicians. Volume II, 882--898, Hindustan Book Agency, New Delhi, 2010. 

\bibitem{segal}
G. Segal, \emph{Equivariant K-theory.} Inst. Hautes {\'E}tudes Sci. Publ. Math. (1968), no. 34, 129-151.

\bibitem{seto}
S. Seto, \emph{On the asymptotic expansion of the Bergman kernel. } PhD Thesis, University of California, Irvine (2015), 100 pp.

\bibitem{seyyedali-szekelyhidi}
R. Seyyedali and G.  Sz\'ekelyhidi, \emph{Extremal metrics on blowups along submanifolds.} J. Differential Geom. 114 (2020), no. 1, 171--192. 

\bibitem{zak}
Z. Sj\"ostr\"om Dyrefelt, \emph{K-semistability of cscK manifolds with transcendental cohomology class.} J. Geom. Anal. 28 (2018), no. 4, 2927--2960.

\bibitem{stoppa}
J. Stoppa, \emph{K-stability of constant scalar curvature K\"ahler manifolds.} Adv. Math. 221 (2009), no. 4, 1397--1408. 

\bibitem{stoppa-szekelyhidi}
J. Stoppa and G. Sz\'ekelyhidi, \emph{Relative K-stability of extremal metrics.} J. Eur. Math. Soc. (JEMS) 13 (2011), no. 4, 899--909. 


\bibitem{gabor-extremal}
G. Sz\'ekelyhidi, \emph{Extremal metrics and K-stability.} Bull. Lond. Math. Soc. 39 (2007), no. 1, 76--84.

\bibitem{gabor-deformations}
G. Sz{\'e}kelyhidi, \emph{The K\"ahler-Ricci flow and K-polystability.} Amer. J. Math. 132 (2010), no. 4, 1077--1090.

\bibitem{gabor-blowups1}
G. Sz{\'e}kelyhidi, \emph{On blowing up extremal {K}\"ahler manifolds}. Duke Math. J. 161 (2012), no. 8, 1411--1453. 

\bibitem{gabor-icm}
G. Sz{\'e}kelyhidi, \emph{Extremal Kähler metrics}. Proceedings of the International Congress of Mathematicians—Seoul 2014. Vol. II, 1017--1032, Kyung Moon Sa, Seoul, 2014. 

\bibitem{gabor-book}
G. Sz\'ekelyhidi, \emph{An introduction to extremal K\"ahler metrics}  Graduate Studies in Mathematics, 152. American Mathematical Society, Providence, RI (2014). xvi+192 pp.

\bibitem{gabor-blowups2}
G. Sz{\'e}kelyhidi, \emph{Blowing up extremal K\"ahler manifolds II}. Invent. Math. 200 (2015), no. 3, 925--977. 

\bibitem{teleman}
A. Teleman, \emph{Symplectic stability, analytic stability in non-algebraic complex geometry.} Internat. J. Math.15 (2004), no.2, 183--209.


\bibitem{tian}
G. Tian, \emph{K\"ahler-Einstein metrics with positive scalar curvature}. Invent. Math. 130 (1997), no. 1, 1--37.


\bibitem{wang}
X. Wang, \emph{GIT stability, K-stability and the moduli space of Fano varieties.} Moduli of K-stable varieties, Springer INdAM Ser., 31, 153--181. 

\bibitem{yau}
S.-T. Yau, \emph{Open problems in geometry.} Differential geometry: partial differential equations on  manifolds ({L}os {A}ngeles, {CA}, 1990), Proc. Sympos. Pure Math., Vol. 54, 1--28.


\end{thebibliography}
\end{document}